\documentclass[12pt]{amsart}
\usepackage{amsfonts}
\usepackage{amssymb, amscd, amsthm}
\usepackage{latexsym}
\usepackage{multirow}
\usepackage{amsmath}
\usepackage[all]{xy}
\usepackage{mathrsfs}
\usepackage{amsbsy}
\usepackage{amsfonts}
\usepackage{mathrsfs}
\usepackage{verbatim}
\usepackage{version}
\usepackage{color}

\newtheorem{theorem}{Theorem}[section]
\newtheorem{lemma}[theorem]{Lemma}
\newtheorem{thm}[theorem]{Theorem}
\newtheorem{prop}[theorem]{Proposition}
\newtheorem{rem}[theorem]{Remark}
\newtheorem{coro}[theorem]{Corollary}
\newtheorem{defn}[theorem]{Definition}
\newtheorem{example}[theorem]{Example}

\newtheorem{con}[theorem]{Conjecture}

\newcommand{\ra}{\rightarrow}
\newcommand{\mo}{\mathcal{O}}
\newcommand{\mf}{\mathcal{F}}
\newcommand{\mg}{\mathcal{G}}
\newcommand{\ma}{\mathcal{A}}
\newcommand{\mb}{\mathcal{B}}
\newcommand{\me}{\mathcal{E}}
\newcommand{\mi}{\mathcal{I}}

\newcommand{\mt}{\mathcal{T}}

\newcommand{\cl}{\mathcal{L}}
\newcommand{\ts}{\mathsf{S}}
\newcommand{\rs}{\mathbf{S}}
\newcommand{\ri}{\mathbf{I}}
\newcommand{\rj}{\mathbf{J}}
\newcommand{\rii}{\mathbf{II}}

\def\<{\langle}
\def\>{\rangle}

\newcommand{\p}{\mathbb{P}}
\newcommand{\bz}{\mathbb{Z}}
\newcommand{\bq}{\mathbb{Q}}
\newcommand{\bc}{\mathbb{C}}

\begin{document}
\fontsize{12pt}{14pt} \textwidth=14cm \textheight=21 cm
\numberwithin{equation}{section}
\title{Rank zero Segre integrals on Hilbert schemes of points on surfaces.}
\author{Yao Yuan}
\thanks{The author is supported by NSFC 21022107.  }

\begin{abstract}We prove the conjecture of Marian-Oprea-Pandharipande on the Segre series associated to a rank zero class.  Hence the rank zero Segre integrals on the Hilbert schemes of points for all surfaces are determined. 

~~~

\end{abstract}

\maketitle
\tableofcontents
\section{Introduction.}

\subsection{Segre classes on the Hilbert scheme of points.}
For any coherent sheaf $\mf$ on a smooth projective variety $X$, the \emph{Segre polynomial} $s_t(\mf):=1+\sum_{i\geq 1}s_i(\mf)t^i$ of $\mf$ is by definition the inverse of its Chern polynomial $c_t(\mf)=1+\sum_{i\geq 1}c_i(\mf)t^i$, i.e. $s_t(\mf)=c_t(-\mf)$.  $s_i(\mf)$ is called the \emph{$i$-th Segre class} of $\mf$.  By an direct computation we have that
\[c_1(\mf)=-s_1(\mf),~c_2(\mf)=s_1(\mf)^2-s_2(\mf),\cdots,c_n(\mf)=-s_1(\mf)c_{n-1}(\mf)-\cdots-s_n(\mf).\]

Let $S$ be a projective surface over the complex number $\bc$.  Denote by $S^{[n]}$ the Hilbert scheme of $n$ points on $S$.  
We have a universal family $\Sigma_n=\{(\xi,x)|x\in\xi\}\subset S^{[n]}\times S$.
Denote by $p_n,q_n$ the projections from $S^{[n]}\times S$ to $S^{[n]}$ and $S$ respectively.  

Let $K(X)$ be the Grothendieck group of coherent sheaves on $X$.  We define a group homomorphism as follows
\begin{eqnarray}\label{ghs}K(S)&\ra&K(S^{[n]})\nonumber\\
\alpha&\mapsto&\alpha^{[n]}:=R(p_n)_*(q_n^*\alpha\cdot\mo_{\Sigma_n}),
\end{eqnarray}
where $\mf\cdot\mg:=\sum_{i\geq0}(-1)^iTor^i(\mf,\mg)$ and $Rf_*\mf=\sum_{i\geq 0}(-1)^iR^if_*\mf$.  If $\alpha$ is the class of a vector bundle $\me$ of rank $r$ over $S$, then $\alpha^{[n]}$ is the class of $(p_n)_*(q_n^*\me|_{\Sigma_n})$ which is a vector bundle of rank $nr$ over $S^{[n]}$.

For any $\alpha\in K(S)$, the \emph{Segre series} $\ts_{\alpha}(z)$ associated to $\alpha$ is defined as follows.
\begin{equation}\ts_{\alpha}(z)=\sum_{n=0}^{\infty}z^n\int_{S^{[n]}}s_{2n}(\alpha^{[n]}).\end{equation}
Let $K_S$ be the canonical divisor class on $S$.  Denote by $c_i(\alpha)$ ($r(\alpha)$, $\chi(\alpha)$, resp.) the $i$-th Chern class (rank, Euler characteristic, resp.) of $\alpha$.  
By \cite{EGL} we have 
\begin{thm}\label{egl}There are five universal series $A_0,A_1,A_2,A_3,A_4\in \bq[[z]]$ independent of the surface $S$ and depending on $\alpha$ only through the rank $r(\alpha)$, such that 
\[\ts_{\alpha}(z)=A_0(z)^{c_2(\alpha)}\cdot A_1(z)^{c_1(\alpha)^2}\cdot A_2(z)^{\chi(\mo_S)}\cdot A_3(z)^{c_1(\alpha).K_S}\cdot A_4(z)^{K_S^2}.\]
\end{thm}    

By Theorem 1 and Remark 6 in \cite{MOP3}, we have for $r(\alpha)=0$ and $z=t(1+t)$
\begin{equation}\label{aaaa}A_0(z)=(1+t)^{-1},~~ A_1(z)=(1+2t)^{\frac12},~~A_2(z)=1,~~A_4(z)=1.
\end{equation}
The following conjecture is due to Marian-Oprea-Pandharipande.
\begin{con}[Conjecture 2 in \cite{MOP3}]\label{mop3}In rank 0, we have
\[A_3(z)=(1+t)^{-1}\cdot (1+2t)^{\frac12}=A_0(z)\cdot A_1(z)\text{ for }z=t(1+t).\]
\end{con}

\subsection{Results.}
Our main result is as follows.
\begin{thm}\label{intro1}Let $S=\p^2_{\bc}$, let $C\in |-K_S|$ be a curve of class $-K_S$.  Define $\alpha\in K(S)$ to be the class of the structure sheaf $\mo_{C}$ of $C$.  Then for all $n\geq 1$ 
$$\int_{S^{[n]}}s_{2n}(\alpha^{[n]})=0.$$ 
\end{thm}

We get the following corollary which proves Conjecture \ref{mop3}.

\begin{coro}\label{intro2}Let $S$ be a projective surface and $\alpha\in K(S)$ with $r(\alpha)=0$.  Then we have  
\[\ts_{\alpha}(z)=(1+t)^{-c_2(\alpha)-c_1(\alpha).K_S}\cdot (1+2t)^{\frac{c_1(\alpha)^2+c_1(\alpha).K_S}2}.\]
where $z=t(1+t)$.

In particular, if $c_1(\alpha)=[C]$ for some smooth curve $C\subset S$, then
\[\ts_{\alpha}(z)=\left(\frac{1+2t}{1+t}\right)^{g_C-1}\cdot(1+t)^{\chi(\alpha)}.\]
\end{coro}
\begin{proof}By (\ref{aaaa}) we only need to show that $A_3(z)=A_0(z)\cdot A_1(z)$ for rank zero. 

Let $S=\p^2$.  Let $\alpha$ be the class of the structure sheaf of a curve in $|-K_S|$.  Then $-c_1(\alpha).K_S=c_1(\alpha)^2=c_2(\alpha)=9$.  By Theorem \ref{egl} and Theorem \ref{intro1} we have
\[\ts_{\alpha}(z)=A_0(z)^{9}\cdot A_1(z)^{9}\cdot A_3(z)^{-9}=1.\]
Therefore $A_3(z)=\mu A_0(z)A_1(z)$ for some $\mu\in\bc$ such that $\mu^9=1$. 

Let $S=\widehat{\p^2}$ be the blow-up of $\p^2$ at a point.  Let $\alpha$ be the class of the structure sheaf of a curve in $|-K_S|$.  Then $-c_1(\alpha).K_S=c_1(\alpha)^2=c_2(\alpha)=8$.  By Theorem \ref{egl} we have
\[\ts_{\alpha}(z)=A_0(z)^{8}\cdot A_1(z)^{8}\cdot A_3(z)^{-8}=\mu^{-8}.\]
Since $\int_{S^{[0]}}s_{0}(\alpha^{[0]})=1$, we have $\mu^{-8}=1\Rightarrow \mu=\mu^9\cdot\mu^{-8}=1$.
Therefore the corollary is proved.
\end{proof}

\subsection{Notations \& Conventions.}\label{NC}
\begin{itemize}
\item For any projective smooth scheme $X$, denote by $A^*(X)=\oplus_{r\geq 0}A^r(X)$ the Chow ring of $X$ and denote by $K(X)$ the Grothendieck group of coherent sheaves over $X$

\item
Except otherwise stated, we always let $S=\p^2_{\bc}$ with $K_S$ the canonical divisor class and $H$ the hyperplane class.  Hence $K_S=-3H$.  

\item By abuse of notations, we use the same letter for any closed subvariety of $S$ and its class in $A^*(S)$.  We have $A^*(S)\cong \bz \mathbf{1}\oplus \bz H\oplus \bz x$ with $x\in \p^2$.  

\item Let $S^{[n]}$  be the Hilbert scheme of $n$ points on $S$.  It is well-known that $A^*(S^{[n]})\cong H^*(S^{[n]},\bz)$ as rings (see e.g. \cite{mark}).   

\item We fix an isomorphism $A^{2n}(S^{[n]})\xrightarrow{\cong}\bz$ such that the image of every class $\theta\in A^{2n}(S^{[n]})$ which we also denote by $\theta$ is equal to $\int_{S^{[n]}}\theta$. 



\item For any sheaf $\mf$ and $i\geq 1$, let $c_i(\mf)$ ($s_i(\mf)$, resp.) be its $i$-th Chern (Segre, resp.) class.  We also define $c_0(\mf)=1=s_0(\mf)$.

\item For two (classes of) sheaves $\mf,\mg\in K(X)$ with $X$ projective smooth, we define their flat tensor 
\begin{equation}\label{flten}\mf\cdot\mg:=\sum_{i\geq0}(-1)^iTor^i(\mf,\mg)\in K(X).\end{equation}
The multiplication $\cdot$ in (\ref{flten}) defines a commutative ring structure on $K(X)$.

\item For any map $f:X\ra Y$ with $X,Y$ smooth projective, let $f^!:K(Y)\ra K(X)$ be the flat pull-back, i.e. $f^!\mf:=f^*\mf$ for any $\mf$ locally free and $f^!$ is a ring homomorphism.  Moreover for any $\mg\in K(Y)$, we have
$$c_i(f^!(\mg))=f^*c_i(\mg),~s_i(f^!(\mg))=f^*s_i(\mg).$$
If $f$ is flat, then $f^!=f^*$.

\item For a multiplication $a_1\cdots a_m$, we let
\[a_1\cdots \widehat{a_k}\cdots a_m:=a_1\cdots a_{k-1}a_{k+1}\cdots a_m.\] 

\item We make the convention that for $k=-1$, $\prod_{j=0}^k P_j=1$ and $\sum_{j=0}^k P_j=0$ with $P_j$ any numbers, classes, or maps etc.
\end{itemize}

The strategy to compute the Serge integral $\int_{S^{[n]}}s_{2n}(\alpha^{[n]})=0$ is essentially to do the induction on $n$, but the whole computation is technic and complicated.

The structure of the paper is arranged as follows.  In \S 2 we review the induction strategy in \cite{EGL} and give some preliminary results.  In \S 3 we prove more recursion relations and at the end of \S 3 we pose Theorem \ref{main} which proves Theorem \ref{intro1}.  In \S 4 we prove Theorem \ref{main}.  \S 4.1 and \S 4.2 are still devoted to some technic propositions and our final proof is in \S 4.3.  

\subsection{Acknowledgements.} 
I thank the referees in advance for their attention to read the paper.

\section{The geometric set-up on Hilbert schemes of points.}
We review the set-up in \cite{EGL} (which in fact applies to $S$ being any projective complex surface).  
Let $\mi_n$ be the ideal sheaf of the universal family $\Sigma_n\subset S^{[n]}\times S$.

Define $\p(\mi_n):=Proj(Sym^*(\mi_n))$.  Then $\p(\mi_n)$ is isomorphic to the incidence variety $S^{[n,n+1]}$ parametrizing all pairs $(\xi,\xi')\in S^{[n]}\times S^{[n+1]}$ satisfying $\xi\subset\xi'$.  We have two classifying morphisms $\psi_n:\p(\mi_n)\ra S^{[n+1]},(\xi,\xi')\mapsto \xi'$ and $\phi_n:\p(\mi_n)\ra S^{[n]},(\xi,\xi')\mapsto \xi$.

It is known that $\p(\mi_n)$ is irreducible and smooth (see e.g. \cite{Che},\cite{Tik},\cite{ES}).

We have the projection $\sigma_n=(\phi_n,\rho_n):\p(\mi_n)\ra S^{[n]}\times S$ with tautological quotient line bundle $\mo_{\sigma_n}(1)=:\cl_n$.  Then $\pi_{*}\cl_n\cong \mi_n$.  
We have the following commutative diagram
\begin{equation}\label{pcd}\xymatrix@C=1.2cm{\p(\mi_n)\cong S^{[n,n+1]}\ar[d]_{\psi_n\qquad}\ar[r]^{\qquad\sigma_n}\ar[rd]^{\phi_n}\ar@/^2pc/[rr]^{\rho_n}&S^{[n]}\times S\ar[d]^{p_n}\ar[r]^{\quad q_n}&S\\ S^{[n+1]}& S^{[n]}&}
\end{equation}

Denote by $\ell_n:=c_1(\cl_n)$.  By Lemma 1.1 in \cite{EGL} we have 
\begin{equation}\label{psl}(\sigma_n)_*((-\ell_n)^i)=c_i(\mo_{\Sigma_n})=s_i(\mi_n).\end{equation}

\begin{lemma}[Lemma 2.1 in \cite{EGL}]\label{ind}The following relation holds in $K(\p(\mi_n))$
\[\psi_n^!\alpha^{[n+1]}=\phi_n^!\alpha^{[n]}+\cl_n\cdot\rho_n^!\alpha\]
for any $\alpha\in K(S)$.
\end{lemma} 

By Lemma \ref{ind} we have
\begin{equation}\label{indse}s_t(\phi_n^!\mf^{[n]})=s_t(\psi_n^!\mf^{[n+1]})c_t(\cl_n\cdot\rho_n^!\mf),\end{equation}
\begin{equation}c_t(\phi_n^!\mf^{[n]})=c_t(\psi_n^!\mf^{[n+1]})s_t(\cl_n\cdot\rho_n^!\mf).\end{equation}

We have a commutative diagram
\begin{equation}\label{pcd2}\xymatrix@C=2cm{\p(\mi_n)&\p(\mi_n)\times S\ar[l]_{p}\ar[d]_{\psi_n^S:=\psi_n\times id_S}\ar[r]^{\sigma_n^S:=\sigma_n\times id_S}\ar[rd]^{\phi_n^S:=\phi_n\times id_S}\ar@/^2pc/[rr]^{\rho_n^S:=\rho_n\times id_S}&S^{[n]}\times S\times S\ar[d]^{p_n^S:=p_n\times id_S}\ar[r]^{\quad q_n^S:=q_n\times id_S}&S\times S\\ &S^{[n+1]}\times S& S^{[n]}\times S&}
\end{equation}

Denote by $\Delta\subset S\times S$ the diagonal.  By \cite{EGL} we have
\begin{equation}\label{exid}0\ra (\psi^S_n)^*\mi_{n+1}\ra(\phi_n^S)^*\mi_n\ra p^*\cl_n\otimes (\rho_n^S)^*\mo_{\Delta}\ra 0.\end{equation}

\begin{lemma}\label{flat}(1) The map $\rho_n^S$ is flat.

(2) $(\psi^S_n)^*\mi_{n+1}\cong (\psi^S_n)^!\mi_{n+1}$ and $(\phi_n^S)^*\mi_n\cong (\phi_n^S)^!\mi_n$.
\end{lemma}
\begin{proof}To show (1), it is enough to show $\rho_n$ is flat.  Since both $\p(\mi_n)$ and $S$ are projective and smooth varieties over $\bc$, by generic smoothness (Corollary 10.7 in Ch III in \cite {Hart}) $\rho_n^S$ is generically smooth.  

On the other hand, any automorphism $g:S\xrightarrow{\cong}S$ induces an isomorphism $g':\p(\mi_n)\xrightarrow{\cong}\p(\mi_n)$ such that we have the commutative diagram
\begin{equation}\label{autcom}\xymatrix{\p(\mi_n)\ar[r]^{g'}\ar[d]_{\rho_n}&\p(\mi_n)\ar[d]^{\rho_n}\\ S\ar[r]_g&S}.\end{equation}
Let $V\subset S$ be a nonempty open subscheme such that $\rho_n$ is flat over $V$.  By (\ref{autcom}), $\rho_n=g^{-1}\circ \rho_n\circ g'$ is flat over $g^{-1}(V)$.  For every point $x\in S=\p^2$, we can find an automorphism $g$ such that $g(x)\in V$.  Hence $\rho_n$ is flat. (1) is proved.

Take a resolution of $\mi_{n+1}$ as follows
\begin{equation}\label{min1}0\ra\ma\ra\mb\ra\mi_{n+1}\ra 0,
\end{equation}
where $\ma,\mb$ are locally free sheaves.
We have 
\begin{equation}\label{pmin1}(\psi_n^S)^*\ma\xrightarrow{h}(\psi_n^S)^*\mb\ra(\psi_n^S)^*\mi_{n+1}\ra 0.
\end{equation}
We want to show that $h$ in (\ref{pmin1}) is injective.  Easy to see that $\psi_n^S$ is generic quasi-finite hence finite since it is projective.  Thus $\psi_n^S$ is generic flat.  Therefore $h$ is injective generically and hence it has to be injective because $(\psi_n^S)^*\ma$ is locally free.

Therefore we have 
\begin{eqnarray}(\psi_n^S)^*\mi_{n+1}&=&(\psi_n^S)^*\mb-(\psi_n^S)^*\ma\nonumber\\
&=& (\psi_n^S)^!\mb-(\psi_n^S)^!\ma=(\psi_n^S)^!(\mb-\ma)=(\psi_n^S)^!\mi_{n+1}.\nonumber\end{eqnarray}

By analogous argument we get $(\phi_n^S)^*\mi_n\cong (\phi_n^S)^!\mi_n$.  (2) is proved.
\end{proof}

By (\ref{exid}) and Lemma \ref{flat} we have
\begin{equation}\label{indin}s_t((\psi_n^S)^!\mi_{n+1})=s_t((\phi_n^S)^!\mi_{n})c_t(\cl_n\cdot(\rho_n^S)^!\mo_{\Delta}),\end{equation}
\begin{equation}c_t((\psi_n^S)^!\mi_{n+1})=c_t((\phi_n^S)^!\mi_{n})s_t(\cl_n\cdot(\rho_n^S)^!\mo_{\Delta}).\end{equation}

By K\"unneth formula we have 
\begin{equation}\label{Kde}H^k(S\times S^{[n]},\bz)\cong \bigoplus_{i+j=k}H^i(S,\bz)\otimes H^j(S^{[n]},\bz)\cong \bigoplus_{i+j=k}A^i(S)\otimes A^j(S^{[n]}).\end{equation}
Recall $A^*(S)=\bz \mathbf{1}\oplus\bz H\oplus \bz x$ with $H$ the hyperplane class and $x\in S$ a point.  Write
\begin{equation}\label{dese}c_i(\mo_{\Sigma_n})=s_i(\mi_n)=\theta_i^0(n)\otimes \mathbf{1}+\theta_i^{1}(n)\otimes H+\theta_i^2(n)\otimes x,
\end{equation}
with $\theta_i^0(n)\in A^i(S^{[n]})$, $\theta_i^{1}(n)\in A^{i-1}(S^{[n]})$ and $\theta_i^2(n)\in A^{i-2}(S^{[n]})$.

The following lemma applies to $S$ being any projective surface.  We won't use this lemma for the proof of our main theorem but still we write it down. 
\begin{lemma}\label{0term}For the decomposition in (\ref{dese}), we have $\theta_i^0(n)=c_i(\mo_x^{[n]})$.
\end{lemma}
\begin{proof}Since $q_n$ is flat, by definition $\mo_x^{[n]}=R(p_n)_*(q_n^!\mo_x\cdot\mo_{\Sigma_n})$.  Denote by $\imath:\{x\}\times S^{[n]}\hookrightarrow S\times S^{[n]}$ the inclusion.  Then $p_n\circ\imath:\{x\}\times S^{[n]}\xrightarrow{\cong} S^{[n]}$ is an isomorphism.  Moreover $\mo_x^{[n]}=(p_n\circ \imath)_*(\imath^!\mo_{\Sigma_n})$ and 
\begin{eqnarray}c_i(\mo_x^{[n]})&=&(p_n\circ \imath)_*\imath^*(c_i(\mo_{\Sigma_n}))\nonumber\\
&=&(p_n\circ \imath)_*\imath^*(s_i(\mi_{n}))\nonumber\\
&=&(S^{[n]}\otimes x).s_i(\mi_{n})=\theta_i^0(n).\end{eqnarray}
The lemma is proved.
\end{proof}

\section{Some recursion relations.}
Let $\theta_i^j(n)$ be as in (\ref{dese}).  We make the convention that $\theta_i^j(n)=0$ for $i<0$.  We have the following proposition.
\begin{prop}\label{rrth}For any $n\geq 0$, we have
\[\psi_n^*\theta_i^0(n+1)=\Omega_i^0(n)+\Omega_i^2(n)\cdot\rho_n^*x,\]
\[\psi_n^*\theta_i^1(n+1)=\Gamma_i^0(n)+\Gamma_i^1(n)\cdot \rho_n^*H+\Gamma_i^2(n)\cdot\rho_n^*x,\]
where 
\begin{eqnarray}\Omega_i^0(n)&:=&\phi_n^*\theta_i^0(n); \nonumber\\ \Omega_i^2(n)&:=&-\sum_{k\geq 0}\phi_n^*\theta_{i-k-2}^0(n)(-\ell_n)^{k}(k+1);\nonumber \\
\Gamma_i^0(n)&:=&\phi_n^*\theta_i^1(n);\nonumber \\
\Gamma_i^1(n)&:=&-\sum_{k\geq 0}\phi_n^*\theta_{i-k-2}^0(n)(-\ell_n)^{k}(k+1);\nonumber\\
\Gamma_i^2(n)&:=&-3\sum_{k\geq 0}\phi_n^*\theta_{i-k-3}^0(n)(-\ell_n)^{k}\frac{(k+1)(k+2)}2\nonumber\\ &-&\sum_{k\geq 0}\phi_n^*\theta_{i-k-2}^1(n)(-\ell_n)^{k}(k+1) \end{eqnarray}
\end{prop}
\begin{proof}By Beilinson spectral sequence we have on $S\times S$
\begin{equation}\label{bss}0\ra \begin{array}{c}p_1^*\mo(-1)\\ \otimes \\p_2^*\mo(-2)\end{array}\ra 
\begin{array}{c}p_1^*\mt_S(-2)\\ \otimes \\p_2^*\mo(-1)\end{array}\ra\mo_{S\times S}\ra\mo_{\Delta}\ra 0,\end{equation}
where $p_i:S\times S\ra S$ is the $i$-th projection, $\mt_S$ is the tangent bundle over $S$ and $\mo(-d):=\mo_S(-dH)$. 

By (\ref{bss}) we can compute the Chern polynomial $c_t(\mo_{\Delta})$ as follows.
\begin{equation}\label{cpdel}c_t(\mo_{\Delta})=\frac{1\cdot(1-p^*_1Ht-p_2^*Ht)}{(1-\gamma_1t-p_2^*Ht)(1-\gamma_2t-p_2^*Ht)}\end{equation}
where $\gamma_1+\gamma_2=p_1^*H$ and $\gamma_1\cdot\gamma_2=p_1^*x$.

Define $H_i=(\rho_n^S)^*p_i^*H, i=1,2$.  Then $H_i^2=(\rho_n^S)^*p_i^*x$ and $H_i^k=0$ for $k\geq 3$.  By (\ref{indin}) we have
\begin{eqnarray}&&(\psi_n^S)^*s_t(\mi_{n+1})\nonumber\\ &=&(\phi_n^S)^*s_t(\mi_n)\cdot c_t(\cl_n\cdot (\rho_n^S)^!\mo_{\Delta})\nonumber \\
&=&(\phi_n^S)^*s_t(\mi_n)\cdot\frac{(1+\ell t)(1+\ell_n t-H_1t-H_2t)}{(1+\ell_n t-(\rho_n^S)^*\gamma_1t-H_2t)(1+\ell_n t-(\rho_n^S)^*\gamma_2t-H_2t)}\nonumber\\
&=&(\phi_n^S)^*s_t(\mi_n)\nonumber\\&+&\frac{(\phi_n^S)^*s_t(\mi_n)(H_1^2+H_2^2+H_1H_2)t^2}{(1-H_1t-2H_2t+\ell_n t+\ell_n^2t^2-\ell_n(H_1+H_2)t^2+(H_1^2+H_2^2+H_1H_2)t^2}.\nonumber\\\end{eqnarray}
On the other hand we have K\"unneth decompositons
\[(\psi_n^S)^*s_i(\mi_{n+1})=\psi_n^*\theta^0_i(n+1)\otimes \mathbf{1}+\psi_n^*\theta^1_i(n+1)\otimes H_2+\psi_n^*\theta_i^2(n+1)\otimes H_2^2;\]
\[(\phi_n^S)^*s_i(\mi_{n})=\phi_n^*\theta^0_i(n)\otimes \mathbf{1}+\phi_n^*\theta^1_i(n)\otimes H_2+\phi_n^*\theta_i^2(n)\otimes H_2^2.\]
Since $p_1\circ\rho_n^S=\rho_n\circ p$, we also have for any $i,j,k\geq 0$,
$$\phi_n^*\theta_i^j(n)\cdot H_1^k=\sigma_n^*(\theta_i^j(n)\otimes q_n^*H^k)=\phi_n^*(\theta_i^j(n))\cdot \rho_n^* H^k.$$  The proposition follows from a direct computation.
\end{proof}

Let $\alpha\in K(S)$ be the class of the structure sheaf $\mo_C$ of a curve $C\in |dH|$ on $S$.  Define $\rs_i(n):=s_i(\alpha^{[n]})\in A^i(S^{[n]})$.  We make the convention that $\rs_i(n)=0$ for $i<0$.

\begin{prop}\label{rrse}For any $n\geq 0$, we have
\[\psi_n^*\rs_i(n+1)=\Lambda_i^0(n)+\Lambda_i^1(n)\cdot\rho_n^*H,\]
where 
\begin{eqnarray}\Lambda_i^0(n)&:=&\phi_n^*\rs_i(n); \nonumber\\ \Lambda_i^1(n)&:=&-d\sum_{k\geq 0}\phi_n^*\rs_{i-k-1}(n)(-\ell_n)^{k}.\nonumber \end{eqnarray}
\end{prop}
\begin{proof}By (\ref{indse}) we have
\begin{eqnarray}\label{secn}s_t(\psi_n^!\alpha^{[n+1]})&=&s_t(\phi_n^!\alpha^{[n]})c_t(\cl_n\cdot \rho_n^!\alpha)^{-1}\nonumber\\
&=&s_t(\phi_n^!\alpha^{[n]})\frac{1+\ell_nt-\rho_n^*dHt}{1+\ell_nt}\nonumber\\
&=&s_t(\phi_n^!\alpha^{[n]})-s_t(\phi_n^!\alpha^{[n]})\frac{\rho_n^*dHt}{1+\ell_nt}.\end{eqnarray}
Therefore 
\begin{equation}\label{sio}\psi_n^*s_i(\alpha^{[n+1]})=\phi_n^*s_i(\alpha^{[n]})-\rho_n^*dH.\left(\sum_{k\geq 0}(-1)^k\ell_n^k\phi_n^*s_{i-k-1}(\alpha^{[n]})\right).\end{equation}
Hence the proposition.
\end{proof}

Define $f_{n+1,n}:=(\phi_n)_*\circ\psi_n^*:A^*(S^{[n+1]})\ra A^*(S^{[n]})$.  Then $f_{n+1,n}$ is a group homomorphism but not a ring homomorphism in general.  Define $$f:=\bigoplus_{n\geq 0} f_{n+1,n}:\bigoplus_{n\geq 0} A^*(S^{[n]})\ra\bigoplus_{n\geq 0}A^*(S^{[n]}),$$
with $f|_{A^*(S^{[0]})}=0$.
Hence 
$$f^k|_{A^*(S^{[n]})}:=f_{n-k+1,n-k}\circ\cdots\circ f_{n,n-1}.$$ 
For simplicity we write $f^k$ instead of $f^k|_{A}$ for any $A\subset \bigoplus_{n\geq 1} A^*(S^{[n]})$.

For any $0\leq i\leq w$, denote by $B(n)_i^w\subset A^*(S^{[n]})$ the \emph{subgroup} generated by all the classes of the form
\[\theta_{l_1}^0(n)\cdots\theta_{l_i}^0(n)\theta_{l_{i+1}}^1(n)\cdots\theta_{l_w}^1(n)\rs_m(n),\]
where $l_1,\cdots,l_w,m$ are nonnegative integers.  Define 
\[B(n)^w=\bigoplus_{0\leq i\leq w}B(n)^w_i;~B(n):=\sum_{w\geq0}B(n)^w.\]
We let $B(n)^w_i=\{0\}$ for $i>w$.
\begin{rem}\label{num}(1) Because $\theta_0^0(n)=1$, we always have for any $k\geq 0$
\[B(n)^w_i\subset B(n)^{w+k}_{i+k}.\]

(2) Let $n=0$, then $\theta_j^1(0)=0$ for all $j$.  Therefore $B(0)_i^w=\{0\}$ if $i\neq w$.
\end{rem}
Define a map for any $k\in\bz_{\geq 0}$ 
\begin{eqnarray}&&[k]:\bigoplus_{n\geq 0}B(n)\ra \bigoplus_{n\geq 0}B(n),\nonumber\\ &&[k](\theta_{l_1}^0(n)\cdots\theta_{l_i}^0(n)\theta_{l_{i+1}}^1(n)\cdots\theta_{l_w}^1(n)\rs_m(n))\nonumber\\
&:=& \theta_{l_1}^0(n-k)\cdots\theta_{l_i}^0(n-k)\theta_{l_{i+1}}^1(n-k)\cdots\theta_{l_w}^1(n-k)\rs_m(n-k).\nonumber\end{eqnarray}
Notice that $[k]|_{B(i)}=0$ if $i<k$.
For simplicity we write $[k]$ instead of $[k]|_{B}$ for any $B\subset\bigoplus_{n\geq k}B(n) $.

From now on we write $p,q,\psi,\phi,\rho,\ell,\theta_i^j$ instead of $p_n,q_n\psi_n,\phi_n,\rho_n,\ell_n,\theta_i^j(n)$ respectively if there is no confusion.  
\begin{lemma}\label{fbim}We have $f(B(n+1)^w)\subset B(n)^{w+1}$.  More precisely we have 
\[f(B(n+1)^w_i)\subset B(n)^{w+1}_i\oplus B(n)^{w+1}_{i+1}\oplus B(n)^{w+1}_{i+2}\oplus B(n)^{w+1}_{i+3}.\]
\end{lemma}
\begin{proof}Let $b:=\theta_{l_1}^0\cdots\theta_{l_i}^0\theta_{l_{i+1}}^1\cdots\theta_{l_w}^1\rs_m(n+1)\in B(n+1)^w_i$.  Then by Proposition \ref{rrth}, Proposition \ref{rrse} and a direct computation we have
\begin{eqnarray}\label{pull}&&\psi^*b=\phi^*(\theta_{l_1}^0\cdots\theta_{l_i}^0\theta_{l_{i+1}}^1\cdots\theta_{l_w}^1\rs_m(n))+\nonumber\\ 
&+&\rho^*H\cdot\left\{\phi^*(\theta_{l_1}^0\cdots\theta_{l_i}^0\theta_{l_{i+1}}^1\cdots\theta_{l_w}^1)\Lambda^1_m(n)+\right.\nonumber\\
&&\left.+\sum_{k=i+1}^w\phi^*(\theta_{l_1}^0\cdots\theta_{l_i}^0\theta_{l_{i+1}}^1\cdots \widehat{\theta^1_{l_k}}\cdots\theta_{l_w}^1\rs_m(n))\Gamma_{l_k}^1(n)\right\}\nonumber\\
&+&\rho^*x\cdot\left\{\sum_{k=1}^i\phi^*(\theta_{l_1}^0\cdots \widehat{\theta^0_{l_k}}\cdots\theta_{l_i}^0\theta_{l_{i+1}}^1\cdots\theta_{l_w}^1\rs_m(n))\Omega_{l_k}^2(n)+\right.\nonumber\\
&&+\sum_{k=i+1}^w\phi^*(\theta_{l_1}^0\cdots\theta_{l_i}^0\theta_{l_{i+1}}^1\cdots \widehat{\theta^1_{l_k}}\cdots\theta_{l_w}^1)\Gamma_{l_k}^1(n)\Lambda_m^1(n)+\nonumber\\
&&+\sum_{k=i+1}^w\phi^*(\theta_{l_1}^0\cdots\theta_{l_i}^0\theta_{l_{i+1}}^1\cdots \widehat{\theta^1_{l_k}}\cdots\theta_{l_w}^1\rs_m(n)\Gamma_{l_k}^2(n)+\nonumber\\
&&+\left.\mathop{\sum_{j,k=i+1}^w}_{j<k}\phi^*(\theta_{l_1}^0\cdots \widehat{\theta^1_{l_j}}\cdots \widehat{\theta^1_{l_k}}\cdots\theta_{l_w}^1\rs_m(n))\Gamma_{l_j}^1(n)\Gamma_{l_k}^1(n)\right\}\nonumber
\end{eqnarray}
Since $\phi=p\circ \sigma$, $\rho=q\circ \sigma$ and $\sigma_*(-\ell)^k=s_k(\mi_n)$, we have $$\phi_*((-\ell)^k\rho^*H^i)=p_*(s_k(\mi_n)\cdot q^*H^i),$$
with $i=1,2$ and $k\geq 0$   

By projection formula we have
\begin{eqnarray}\label{push}&&f(b)=\phi_*(\psi^*b)=\nonumber\\
&&\theta_{l_1}^0\cdots\theta_{l_i}^0\theta_{l_{i+1}}^1\cdots\theta_{l_w}^1\left(-d\sum_{t\geq 0}\rs_{m-t-1}(n)\cdot p_*(s_t(\mi_n)\cdot q^*H)\right)\nonumber\\
&+&\sum_{k=i+1}^w\theta_{l_1}^0\cdots\theta_{l_i}^0\theta_{l_{i+1}}^1\cdots \widehat{\theta^1_{l_k}}\cdots\theta_{l_w}^1\rs_m(n)\left(-\sum_{t\geq 0}\theta_{l_k-t-2}^0\cdot p_*(s_{t}(\mi_n)\cdot q^*H)(t+1)\right)\nonumber\\
&+&\sum_{k=1}^i\theta_{l_1}^0\cdots \widehat{\theta^0_{l_k}}\cdots\theta_{l_i}^0\theta_{l_{i+1}}^1\cdots\theta_{l_w}^1\rs_m(n)\left(-\sum_{t\geq 0}\theta_{l_k-t-2}^0\cdot p_*(s_{t}(\mi_n)\cdot q^*x)(t+1)\right)\nonumber\\
&+&\sum_{k=i+1}^w\left\{\theta_{l_1}^0\cdots\theta_{l_i}^0\theta_{l_{i+1}}^1\cdots \widehat{\theta^1_{l_k}}\cdots\theta_{l_w}^1\right.\times\nonumber\\
&&\times\left.\left(d\sum_{t\geq 0}\sum_{a\geq 0}\theta_{l_k-t-2}^0\rs_{m-a-1}(n)\cdot p_*(s_{a+t}(\mi_n)\cdot q^*x)(t+1)\right)\right\}\nonumber\\&+&\sum_{k=i+1}^w\left\{\theta_{l_1}^0\cdots\theta_{l_i}^0\theta_{l_{i+1}}^1\cdots \widehat{\theta^1_{l_k}}\cdots\theta_{l_w}^1\rs_m(n)\times\right.\nonumber\\ &&\times\left.\left(-3\sum_{t\geq 0}\theta_{l_k-t-3}^0\cdot p_*(s_t(\mi_n)\cdot q^*x)\frac{(t+1)(t+2)}2\right)\right\}\nonumber\\
&+&\sum_{k=i+1}^w\left\{\theta_{l_1}^0\cdots\theta_{l_i}^0\theta_{l_{i+1}}^1\cdots \widehat{\theta^1_{l_k}}\cdots\theta_{l_w}^1\rs_m(n)\right.\times\nonumber\\
&&\times\left.\left(-\sum_{t\geq 0}\theta_{l_k-t-2}^1\cdot p_*(s_t(\mi_n)\cdot q^*x)(t+1)\right)\right\}\nonumber\\
&+&\mathop{\sum_{j,k=i+1}^w}_{j<k}\left\{\theta_{l_1}^0\cdots \widehat{\theta^1_{l_j}}\cdots \widehat{\theta^1_{l_k}}\cdots\theta_{l_w}^1\rs_m(n)\right.\times\nonumber\\&&\times\left.\left(\mathop{\sum_{t\geq 0}}_{a\geq 0}\theta_{l_k-t-2}^0\theta_{l_j-a-2}^0\cdot p_*(s_{a+t}(\mi_n)\cdot q^*x)(t+1)(a+1)\right)\right\}.\nonumber
\end{eqnarray}
Because $s_k(\mi_n)=\theta_k^0\otimes\mathbf{1}+\theta_k^1\otimes H+\theta_k^2\otimes x$, hence $p_*(s_t(\mi_n)\cdot q^*H)=\theta_k^1$ and $p_*(s_t(\mi_n)\cdot q^*x)=\theta_k^0$, and hence \footnotesize
\begin{eqnarray}\label{pupu}&&f(b)=\phi_*(\psi^*b)=\theta_{l_1}^0\cdots\theta_{l_i}^0\theta_{l_{i+1}}^1\cdots\theta_{l_w}^1\left(-d\sum_{t\geq 0}\rs_{m-t-1}(n)\cdot \theta_t^1\right)+\nonumber\\
&+&\sum_{k=i+1}^w\theta_{l_1}^0\cdots\theta_{l_i}^0\theta_{l_{i+1}}^1\cdots \widehat{\theta^1_{l_k}}\cdots\theta_{l_w}^1\rs_m(n)\left(-l_k\sum_{t\geq 0}\theta_{l_k-t-2}^0\cdot \theta_t^1\right)\nonumber\\
&+&\sum_{k=1}^i\theta_{l_1}^0\cdots \widehat{\theta^0_{l_k}}\cdots\theta_{l_i}^0\theta_{l_{i+1}}^1\cdots\theta_{l_w}^1\rs_m(n)\left(-\sum_{t\geq 0}\theta_{l_k-t-2}^0\cdot \theta_t^0\cdot (t+1)\right)\nonumber\\
&+&\sum_{k=i+1}^w\theta_{l_1}^0\cdots\theta_{l_i}^0\theta_{l_{i+1}}^1\cdots \widehat{\theta^1_{l_k}}\cdots\theta_{l_w}^1\left(d\sum_{t\geq 0}\sum_{a\geq 0}\theta_{l_k-t-2}^0\rs_{m-a-1}(n)\theta_{a+t}^0\cdot(t+1)\right).\nonumber\\&+&\sum_{k=i+1}^w\theta_{l_1}^0\cdots\theta_{l_i}^0\theta_{l_{i+1}}^1\cdots \widehat{\theta^1_{l_k}}\cdots\theta_{l_w}^1\rs_m(n)\left(-3\sum_{t\geq 0}\theta_{l_k-t-3}^0\cdot \theta^0_t\frac{(t+1)(t+2)}2\right)\nonumber\\
&+&\mathop{\sum_{j,k=i+1}^w}_{j<k}\theta_{l_1}^0\cdots \widehat{\theta^1_{l_j}}\cdots \widehat{\theta^1_{l_k}}\cdots\theta_{l_w}^1\rs_m(n)\left(\mathop{\sum_{t\geq 0}}_{a\geq 0}\theta_{l_k-t-2}^0\theta_{l_j-a-2}^0\theta^0_{a+t}\cdot(t+1)(a+1)\right)\nonumber\\ 
\end{eqnarray}\normalsize
The lemma is proved by a direct observation on (\ref{pupu}).
\end{proof}

Define $f:=\sum_{j=0}^3\delta_j$ such that $\delta_j(B(n+1)^w_i)\subset B(n)^{w+1}_{i+j}$.  Let $b=\theta_{l_1}^0\cdots\theta_{l_i}^0\theta_{l_{i+1}}^1\cdots\theta_{l_w}^1\rs_m(n+1)\in B(n+1)^w_i$.  By (\ref{pupu}) we can describe $\delta_j$ explicitly as follows.

\begin{equation}\label{del0}\delta_0(b)=\theta_{l_1}^0\cdots\theta_{l_i}^0\theta_{l_{i+1}}^1\cdots\theta_{l_w}^1\left(-d\sum_{t\geq 0}\rs_{m-t-1}(n)\cdot \theta_t^1\right);
\end{equation}

\begin{eqnarray}\label{del1}\delta_1(b)&=&\sum_{k=i+1}^w\theta_{l_1}^0\cdots\theta_{l_i}^0\theta_{l_{i+1}}^1\cdots \widehat{\theta^1_{l_k}}\cdots\theta_{l_w}^1\rs_m(n)\left(-l_k\sum_{t\geq 0}\theta_{l_k-t-2}^0\cdot \theta_t^1\right)\nonumber\\
&+&\sum_{k=1}^i\theta_{l_1}^0\cdots \widehat{\theta^0_{l_k}}\cdots\theta_{l_i}^0\theta_{l_{i+1}}^1\cdots\theta_{l_w}^1\rs_m(n)\left(-\sum_{t\geq 0}\theta_{l_k-t-2}^0\cdot \theta_t^0\cdot (t+1)\right);
\end{eqnarray}
\begin{eqnarray}\label{del2}\delta_2(b)&=&\sum_{k=i+1}^w\theta_{l_1}^0\cdots\theta_{l_i}^0\theta_{l_{i+1}}^1\cdots \widehat{\theta^1_{l_k}}\cdots\theta_{l_w}^1\left(d\sum_{t\geq 0}\sum_{a\geq 0}\theta_{l_k-t-2}^0\rs_{m-a-1}(n)\theta_{a+t}^0\cdot(t+1)\right).\nonumber\\&+&\sum_{k=i+1}^w\theta_{l_1}^0\cdots\theta_{l_i}^0\theta_{l_{i+1}}^1\cdots \widehat{\theta^1_{l_k}}\cdots\theta_{l_w}^1\rs_m(n)\left(-3\sum_{t\geq 0}\theta_{l_k-t-3}^0\cdot \theta^0_t\frac{(t+1)(t+2)}2\right);\nonumber\\
\end{eqnarray}
\begin{eqnarray}\label{del3}\delta_3(b)=\mathop{\sum_{j,k=i+1}^w}_{j<k}\theta_{l_1}^0\cdots \widehat{\theta^1_{l_j}}\cdots \widehat{\theta^1_{l_k}}\cdots\theta_{l_w}^1\rs_m(n)\left(\mathop{\sum_{t\geq 0}}_{a\geq 0}\theta_{l_k-t-2}^0\theta_{l_j-a-2}^0\theta^0_{a+t}\cdot(t+1)(a+1)\right).\nonumber\\
\end{eqnarray}
\begin{lemma}\label{lbz}We have
\[\delta_0(b)=[1](\theta_{l_1}^0\cdots\theta_{l_i}^0\theta_{l_{i+1}}^1\cdots\theta_{l_w}^1)\cdot \delta_0(\rs_m(n+1));\]
\begin{eqnarray}\delta_1(b)&=&\delta_1(\theta_{l_1}^0\cdots\theta_{l_i}^0\theta_{l_{i+1}}^1\cdots\theta_{l_w}^1\rs_0(n+1))\cdot \rs_m(n)\nonumber\\
&=&\delta_1(\theta_{l_1}^0\cdots\theta_{l_i}^0\theta_{l_{i+1}}^1\rs_m(n+1))\cdot [1](\theta^1_{l_{i+2}}\cdots\theta_{l_w}^1\rs_0(n+1))\nonumber\\
&&+
\sum_{k=i+2}^w[1](\theta_{l_1}^0\cdots\theta_{l_i}^0\theta_{l_{i+1}}^1\cdots \widehat{\theta^1_{l_k}}\cdots\theta_{l_w}^1\rs_m(n+1))\cdot \delta_1(\theta_{l_k}^1\rs_0(n+1))\nonumber\\
&=&\delta_1(\theta_{i_1}^{j_1}\cdots\theta_{i_s}^{j_s}\rs_0(n+1))\cdot [1](\theta^{j_{s+1}}_{i_{s+1}}\cdots\theta_{i_w}^{j_w}\rs_m(n+1))\nonumber\\
&&+
[1](\theta_{i_1}^{j_1}\cdots\theta_{i_s}^{j_s}\rs_0(n+1))\cdot \delta_1(\theta^{j_{s+1}}_{i_{s+1}}\cdots\theta_{i_w}^{j_w}\rs_m(n+1))\nonumber\\
&=&\delta_1(\theta_{i_1}^{j_1}\cdots\theta_{i_s}^{j_s}\rs_0(n+1))\cdot [1](\theta^{j_{s+1}}_{i_{s+1}}\cdots\theta_{i_w}^{j_w}\rs_m(n+1))\nonumber\\
&&+
[1](\theta_{i_1}^{j_1}\cdots\theta_{i_s}^{j_s}\rs_m(n+1))\cdot \delta_1(\theta^{j_{s+1}}_{i_{s+1}}\cdots\theta_{i_w}^{j_w}\rs_0(n+1)),
\end{eqnarray}
where $s\geq 0$ and $\{\theta_{i_1}^{j_1},\cdots,\theta_{i_w}^{j_w}\}=\{\theta_{l_1}^0,\cdots,\theta_{l_i}^0,\theta_{l_{i+1}}^1,\cdots,\theta_{l_w}^1\}.$

Moreover $[1]\circ\delta_i=\delta_i\circ[1]$ for $i=1,2,3,4.$
\end{lemma}
\begin{proof}Notice that $\rs_0(n)=\mathbf{1}$ is the unit for the multiplication in $A^*(S^{[n]})$.  The lemma is proved by a direct observation.
\end{proof}

Denote by $(d-3)A^*(S^{[n]})$ the subgroup of $A^*(S^{[n]})$ consisting of elements which are $d-3$ times of some elements in $A^*(S^{[n]})$.
\begin{lemma}\label{comm02}$\delta_2\circ\delta_0(b)-\delta_0\circ\delta_2(b)\in (d-3)A^*(S^{[n-1]})$.
\end{lemma}
\begin{proof}
By (\ref{del0}) and (\ref{del2}) we have
\begin{eqnarray}\label{del02}&&\delta_0\circ\delta_2(b)\nonumber\\
&=&\delta_0\left(\sum_{k=i+1}^w\theta_{l_1}^0\cdots\theta_{l_i}^0\theta_{l_{i+1}}^1\cdots \widehat{\theta^1_{l_k}}\cdots\theta_{l_w}^1\left(d\sum_{\substack{t\geq 0\\ a\geq 0}}\theta_{l_k-t-2}^0\rs_{m-a-1}(n)\theta_{a+t}^0\cdot(t+1)\right)\right.\nonumber\\&+&\left.\sum_{k=i+1}^w\theta_{l_1}^0\cdots\theta_{l_i}^0\theta_{l_{i+1}}^1\cdots \widehat{\theta^1_{l_k}}\cdots\theta_{l_w}^1\rs_m(n)\left(-3\sum_{t\geq 0}\theta_{l_k-t-3}^0\cdot \theta^0_t\frac{(t+1)(t+2)}2\right)\right)\nonumber\\
&=&\sum_{k=i+1}^w\theta_{l_1}^0\cdots\theta_{l_i}^0\theta_{l_{i+1}}^1\cdots \widehat{\theta^1_{l_k}}\cdots\theta_{l_w}^1\left(-d^2\sum_{\substack{t\geq 0\\ a\geq 0\\ u\geq 0}}\theta_u^1\theta_{l_k-t-2}^0\rs_{m-a-u-2}(n-1)\theta_{a+t}^0\cdot(t+1)\right)\nonumber\\&+&\sum_{k=i+1}^w\theta_{l_1}^0\cdots\theta_{l_i}^0\theta_{l_{i+1}}^1\cdots \widehat{\theta^1_{l_k}}\cdots\theta_{l_w}^1\left(3d\sum_{\substack{t\geq 0\\ u\geq0}}\theta_u^1\rs_{m-u-1}(n-1)\theta_{l_k-t-3}^0\cdot \theta^0_t\frac{(t+1)(t+2)}2\right)\nonumber\\
\end{eqnarray}
\begin{eqnarray}\label{del20}&&\delta_2\circ\delta_0(b)\nonumber\\
&=&\delta_2\left(\theta_{l_1}^0\cdots\theta_{l_i}^0\theta_{l_{i+1}}^1\cdots\theta_{l_w}^1\left(-d\sum_{t\geq 0}\rs_{m-t-1}(n)\cdot \theta_t^1\right)\right)\nonumber\\
&=&\sum_{k=i+1}^w\theta_{l_1}^0\cdots\theta_{l_i}^0\theta_{l_{i+1}}^1\cdots \widehat{\theta^1_{l_k}}\cdots\theta_{l_w}^1\left(-d^2\sum_{\substack{t\geq 0\\ a\geq 0\\ u\geq 0}}\theta_u^1\theta_{l_k-t-2}^0\rs_{m-a-u-2}(n)\theta_{a+t}^0\cdot(t+1)\right)\nonumber\\&+&\sum_{k=i+1}^w\theta_{l_1}^0\cdots\theta_{l_i}^0\theta_{l_{i+1}}^1\cdots \widehat{\theta^1_{l_k}}\cdots\theta_{l_w}^1\left(3d\sum_{\substack{t\geq 0\\ u\geq0}}\theta_u^1\rs_{m-u-1}(n)\theta_{l_k-t-3}^0\cdot \theta^0_t\frac{(t+1)(t+2)}2\right)\nonumber\\&-&d\cdot\theta_{l_1}^0\cdots\theta_{l_w}^1\delta_2\left(\sum_{u\geq 0}\theta_u^1\rs_{m-u-1}(n)\right).\nonumber\\
\end{eqnarray}
By (\ref{del02}) and (\ref{del20}) we only need to prove 
\begin{equation}\label{del0220}\delta_2\left(\sum_{u\geq 0}\theta_u^1\rs_{m-u-1}(n)\right)\in (d-3) A^*(S^{[n-1]}).\end{equation}
By (\ref{del2}) we have
\begin{eqnarray}&&\delta_2\left(\sum_{u\geq 0}\theta_u^1\rs_{m-u-1}(n)\right)\nonumber\\&=&d\sum_{\substack{u\geq 0\\a\geq 0\\t\geq0}}\theta^0_{u-a-2}\theta_{a+t}^0(a+1)\rs_{m-u-t-2}(n-1)-3\sum_{\substack{u\geq0\\ a\geq 0}}\theta_{u-a-3}^0\theta_a^0\frac{(a+1)(a+2)}2\rs_{m-u-1}(n-1).\nonumber
\end{eqnarray}
Define $u':=u+t+1$ and $a':=a+t$.  Since $\theta_{k}^i=0$ for $k<0$, we then have 
\footnotesize
\begin{eqnarray}&&\delta_2\left(\sum_{u\geq 0}\theta_u^1\rs_{m-u-1}(n)\right)\nonumber\\&=&d\sum_{\substack{u\geq a+2\\a\geq 0\\t\geq0}}\theta^0_{u-a-2}\theta_{a+t}^0(a+1)\rs_{m-u-t-2}(n-1)-3\sum_{\substack{u\geq a+3\\ a\geq 0}}\theta_{u-a-3}^0\theta_a^0\frac{(a+1)(a+2)}2\rs_{m-u-1}(n-1)\nonumber\\
&=&d\sum_{\substack{u'\geq a'+3\\a'\geq 0\\t\geq0}}\theta^0_{u'-a'-3}\theta_{a'}^0(a'-t+1)\rs_{m-u'-1}(n-1)-3\sum_{\substack{u\geq a+3\\ a\geq 0}}\theta_{u-a-3}^0\theta_a^0\frac{(a+1)(a+2)}2\rs_{m-u-1}(n-1)\nonumber\\
&=&d\sum_{\substack{u'\geq a'+3\\a'\geq 0}}\theta^0_{u'-a'-3}\theta_{a'}^0\left(\sum_{t=0}^{a'}(a'-t+1)\right)\rs_{m-u'-1}(n-1)\nonumber\\ 
&&-3\sum_{\substack{u\geq a+3\\ a\geq 0}}\theta_{u-a-3}^0\theta_a^0\frac{(a+1)(a+2)}2\rs_{m-u-1}(n-1)\nonumber\\
&=&d\sum_{\substack{u'\geq a'+3\\a'\geq 0}}\theta^0_{u'-a'-3}\theta_{a'}^0\frac{(a'+1)(a'+2)}2\rs_{m-u'-1}(n-1)\nonumber\\
&&-3\sum_{\substack{u\geq a+3\\ a\geq 0}}\theta_{u-a-3}^0\theta_a^0\frac{(a+1)(a+2)}2\rs_{m-u-1}(n-1)\nonumber\\
&=&(d-3)\sum_{\substack{u\geq a+3\\ a\geq 0}}\theta_{u-a-3}^0\theta_a^0\frac{(a+1)(a+2)}2\rs_{m-u-1}(n-1).\nonumber
\end{eqnarray}
\normalsize
The lemma is proved.
\end{proof}

Recall that $\rs_i(n):=s_i(\alpha^{[n]})\in A^i(S^{[n]})$ with $\alpha\in K(S)$ the class of the structure sheaf $\mo_C$ of a curve $C\in |dH|$ on $S$. 
Theorem \ref{intro1} can be deduced from the following theorem.
\begin{thm}\label{main}For any $k,m,n\geq1$ we have
\begin{equation}\label{tsum}\sum_{\sum_{l=1}^ki_l=k}\delta_{i_1}\circ\cdots\circ\delta_{i_k}(\rs_m(n))\in (d-3)A^*(S^{[n-k]}).\end{equation}
\end{thm}  
\begin{rem}\label{lag}Since $\rs_m(n)\in B(n)_0^0$ and $B(k)_i^w=\{0\}$ for any $i>w$, we have 
\[\delta_{i_1}\circ\cdots\circ\delta_{i_k}(\rs_m(n))=0\]
for any $\sum_{l=1}^ki_l>k$.
\end{rem}
We will prove Theorem \ref{main} in the next section and we now use it to prove Theorem \ref{intro1}.
\begin{proof}[Proof of Theorem \ref{intro1}]Since $\psi:S^{[n,n+1]}\ra S^{[n+1]}$ is generically of degree $n+1$,  we have 
\[f(\rs_{2(n+1)}(n+1))=\phi_*\psi^*(\rs_{2(n+1)}(n+1))=(n+1)\rs_{2(n+1)}(n+1),\]
where we use the same letter for the classes in $A^{2n}(S^{[n]})$ and their image under some fixed isomorphism $A^{2n}(S^{[n]})\cong \bz$ as we have stated in \S \ref{NC}.  Therefore 
\[f^{n+1}(\rs_{2(n+1)}(n+1))=(n+1)!~\rs_{2(n+1)}(n+1).\]
Since $-K_S=3H$, it is enough to show $f^{n+1}(\rs_{2(n+1)}(n+1))\in A^*(S^{[0]})$ is zero for $d=3$.

Because $f=\sum_{j=0}^3\delta_j$, we have
\[f^{n+1}(\rs_{2(n+1)}(n+1))=\sum_{i_1,\cdots,i_{n+1}\in\{0,1,2,3\}}\delta_{i_1}\circ\cdots\circ \delta_{i_{n+1}}(\rs_{2(n+1)}(n+1)).\]
Notice that $\rs_{2(n+1)}(n+1)\in B(n+1)^0_0$.  By the definition of $\delta_j$, we have 
\[\delta_{i_1}\circ\cdots\circ \delta_{i_{n+1}}(\rs_{2(n+1)}(n+1))\in B(0)^{n+1}_{n'},\]
with $n'=\sum_{l=1}^{n+1} i_{l}.$

By Remark \ref{num} (2), we have 
\[f^{n+1}(\rs_{2(n+1)}(n+1))=\sum_{\sum_{l=1}^{n+1}i_l=n+1}\delta_{i_1}\circ\cdots\circ\delta_{i_{n+1}}(\rs_{2(n+1)}(n+1)).\]
By Theorem \ref{main} we have $f^{n+1}(\rs_{2(n+1)}(n+1))=0$ for $d=3$.  The theorem is proved.
\end{proof}

\section{Proof of the main theorem.} 
Theorem \ref{main} will be proved by doing induction on $k$.
From now on for simplicity we write $\delta_{i_i}\delta_{i_2}=\delta_{i_1}\circ\delta_{i_2}$ and $(\delta_{i})^k=\underbrace{\delta_i\circ\cdots\circ\delta_i}_{k~times}$.    

\subsection{A special part of the sum.}
In this subsection we show that a special part of the sum in (\ref{tsum}) belongs to $(d-3)A^*(S^{[n]})$ as stated by the following proposition.  
\begin{prop}\label{wkmain}For any $k,m,n\geq 1$ we have
\[\left(\delta_2(\delta_1)^k\delta_0+\sum_{s=0}^{k-1}\delta_3(\delta_1)^{k-s-1}\delta_0(\delta_1)^{s}\delta_0\right)(\rs_m(n))\in (d-3)A^*(S^{[n-k-2]}).\]
\end{prop}

We need several lemmas to prove Proposition \ref{wkmain}.  Recall that we have made the convention that $\prod_{j=0}^k P_j=1$ for $k=-1$.  Define $\theta^0_{\vec{a},l}:=\theta^0_{a_1}\cdots\theta^0_{a_l}$.  

\begin{lemma}\label{del21S}For any $k\geq 0$ we have $(\delta_1)^k\delta_0(\rs_m(n))=$
\[(-1)^{k+1}d\sum_{\substack{m'+a+2k+1\\+\sum_{i=1}^ka_i=m}}\left\{\prod_{j=0}^{k-1}\left(j+a+1+\sum_{i=1}^k (a_i+1)\right)\right\}\theta_{\vec{a},k}^0\theta_a^1\rs_{m'}(n-k-1).\]
\end{lemma}
\begin{proof}We prove the lemma by doing induction on $k$.  Let $k=0$.  Since $\prod_{j=0}^{k-1}\left(j+a+1+\sum_{i=1}^k (a_i+1)\right)=1$ for $k-1<0$, we get the lemma directly from (\ref{del0}).  

Assume the lemma is true for $k\leq k_0$, then by (\ref{del1}) we have
\footnotesize
\begin{eqnarray}\label{del211}&&\frac{(-1)^{k_0+2}}d(\delta_1)^{k_0+1}\delta_0(\rs_m(n))=\frac{(-1)^{k_0+2}}d\delta_1((\delta_1)^{k_0}\delta_0(\rs_m(n)))\nonumber\\
&=&\delta_1\left(-\sum_{\substack{m'+a+2k_0+1\\+\sum_{i=1}^{k_0}a_i=m}}\left\{\prod_{j=0}^{k_0-1}\left(j+a+1+\sum_{i=1}^{k_0} (a_i+1)\right)\right\}\theta_{\vec{a},k_0}^0\theta_a^1\rs_{m'}(n-k_0-1)\right)\nonumber\\
&=&-\sum_{\substack{m'+a+2k_0+1\\+\sum_{i=1}^{k_0}a_i=m}}\left\{\prod_{j=0}^{k_0-1}\left(j+a+1+\sum_{i=1}^{k_0} (a_i+1)\right)\right\}\delta_1(\theta_{\vec{a},k_0}^0\theta_a^1\rs_{m'}(n-k_0-1))
\nonumber\\
&=&\sum_{\substack{m'+a+2k_0+1\\+\sum_{i=1}^{k_0}a_i=m}}\left\{\prod_{j=0}^{k_0-1}\left(j+a+1+\sum_{i=1}^{k_0} (a_i+1)\right)\right\}\cdot\rs_{m'}(n-k_0-2)\times\nonumber\\
&&\left\{\theta_{\vec{a},k_0}^0(a\sum_{t\geq0}\theta_{a-t-2}^0\theta_t^1)+\sum_{l=1}^{k_0}\theta_{a_1}^0\cdots\widehat{\theta_{a_l}^0}\cdots\theta_{a_{k_0}}^0\theta_{a}^1(\sum_{t\geq0}\theta_{a_l-t-2}^0\theta_t^0(t+1))\right\}\nonumber\\
\end{eqnarray}
\normalsize
By defining $a'=t$, $a'_i=a_i,~i=1,\cdots,k_0$ and $a'_{k_0+1}=a-t-2$, we have
\footnotesize
\begin{eqnarray}\label{del212}&&\sum_{\substack{m'+a+2k_0+1\\+\sum_{i=1}^{k_0}a_i=m}}\left\{\prod_{j=0}^{k_0-1}\left(j+a+1+\sum_{i=1}^{k_0} (a_i+1)\right)\right\}
\theta_{\vec{a},k_0}^0\rs_{m'}(n-k_0-2)(a\sum_{t\geq0}\theta_{a-t-2}^0\theta_t^1)\nonumber\\
&=&\sum_{\substack{m'+a'+2k_0+3\\+\sum_{i=1}^{k_0+1}a'_i=m}}\left\{\prod_{j=1}^{k_0}\left(j+a'+1+\sum_{i=1}^{k_0+1} (a'_i+1)\right)\right\}
\theta_{\vec{a}',k_0+1}^0\theta_{a'}^1(a'+a'_{k_0+1}+2)\rs_{m'}(n-k_0-2)\nonumber\\
\end{eqnarray}
\normalsize
By defining $a'=a$, $a'_l=t$, $a'_i=a_i,1\leq i\leq k_0,~i\neq l$ and $a'_{k_0+1}=a_l-t-2$, we have
\footnotesize
\begin{eqnarray}\label{del213}&&\sum_{\substack{m'+a+2k_0+1\\+\sum_{i=1}^{k_0}a_i=m}}\left\{\prod_{j=0}^{k_0-1}\left(j+a+1+\sum_{i=1}^{k_0} (a_i+1)\right)\right\}
\theta_{a_1}^0\cdots\widehat{\theta_{a_l}^0}\cdots\theta_{a_{k_0}}^0\theta_{a}^1(\sum_{t\geq0}\theta_{a_l-t-2}^0\theta_t^0(t+1))\nonumber\\
&=&\sum_{\substack{m'+a'+2k_0+3\\+\sum_{i=1}^{k_0+1}a'_i=m}}\left\{\prod_{j=1}^{k_0}\left(j+a'+1+\sum_{i=1}^{k_0+1} (a'_i+1)\right)\right\}
\theta_{\vec{a}',k_0+1}^0\theta_{a'}^1(a'_l+1)\rs_{m'}(n-k_0-2)\nonumber\\
\end{eqnarray}
\normalsize
Combine (\ref{del211}), (\ref{del212}) and (\ref{del213}) and we have
\footnotesize
\begin{eqnarray}\label{del211}&&\frac{(-1)^{k_0+2}}d(\delta_1)^{k_0+1}\delta_0(\rs_m(n))=\nonumber\\
&=&\sum_{\substack{m'+a'+2k_0+3\\+\sum_{i=1}^{k_0+1}a'_i=m}}\left\{\prod_{j=1}^{k_0}\left(j+a'+1+\sum_{i=1}^{k_0+1} (a'_i+1)\right)\right\}
\theta_{\vec{a}',k_0+1}^0\theta_{a'}^1\rs_{m'}(n-k_0-2)\times\nonumber\\
&&(a'+1+\sum_{i=1}^{k_0+1}(a'_i+1))\nonumber\\
&=&\sum_{\substack{m'+a'+2k_0+3\\+\sum_{i=1}^{k_0+1}a'_i=m}}\left\{\prod_{j=0}^{k_0}\left(j+a'+1+\sum_{i=1}^{k_0+1} (a'_i+1)\right)\right\}
\theta_{\vec{a}',k_0+1}^0\theta_{a'}^1\rs_{m'}(n-k_0-2)\nonumber
\end{eqnarray}
\normalsize
The lemma is proved.
\end{proof}
\begin{lemma}\label{del31S}For $k\geq1$ and $0\leq s\leq k$, we have
\begin{eqnarray}\label{del101s}&&(\delta_1)^{k-1-s}\delta_0(\delta_1)^s\delta_0(\rs_m(n))\nonumber\\&=&(-1)^{k+1}d^2\sum_{\substack{m'+a+b+2\\+2(k-1)+\sum_{l=1}^{s+i}a_l\\+\sum_{l=1}^{k-s-i-1}b_l=m}}\sum_{i=0}^{k-s-1}\left\{\binom{k-s-1}{i}\times\right.\nonumber\\
&&\times\left(\prod_{j=0}^{k-s-i-2}\left(j+b+1+\sum_{i=1}^{k-s-i-1} (b_i+1)\right)\theta_{\vec{b},k-i-s-1}^0\theta_b^1\right)\nonumber\\
&&\times\left.\left(\prod_{j=0}^{i+s-1}\left(j+a+1+\sum_{i=1}^{s+i} (a_i+1)\right)\theta_{\vec{a},i+s}^0\theta_a^1\right)\right\}\rs_{m'}(n-k-1).\nonumber\\
\end{eqnarray}
In particular 
\begin{eqnarray}\label{del101a}&&\sum_{s=0}^{k-1}(\delta_1)^{k-1-s}\delta_0(\delta_1)^s\delta_0(\rs_m(n))\nonumber\\&=&(-1)^{k+1}d^2\sum_{M=0}^{k-1}\sum_{\substack{m'+a+b+2\\+2(k-1)+\sum_{l=1}^{M}a_l\\+\sum_{l=1}^{k-M-1}b_l=m}}\binom{k}{M}\rs_{m'}(n-k-1)\times\nonumber \\
&&\left\{\left(\prod_{j=0}^{k-M-2}\left(j+b+1+\sum_{i=1}^{k-M-1} (b_i+1)\right)\theta_{\vec{b},k-M-1}^0\theta_b^1\right)\right.\nonumber\\
&&\times\left.\left(\prod_{j=0}^{M-1}\left(j+a+1+\sum_{i=1}^{M} (a_i+1)\right)\theta_{\vec{a},M}^0\theta_a^1\right)\right\}.\nonumber\\
\end{eqnarray}

\end{lemma}
\begin{proof}By Lemma \ref{del21S} and (\ref{del0}) we have
\begin{eqnarray}\label{add1}&&\delta_0(\delta_1)^s\delta_0(\rs_m(n))\nonumber\\
&=&(-1)^{s+1}d\sum_{\substack{\widetilde{m}+a+2s+1\\+\sum_{i=1}^ka_i=m}}\left\{\prod_{j=0}^{s-1}\left(j+a+1+\sum_{i=1}^s (a_i+1)\right)\right\}\theta_{\vec{a},s}^0\theta_a^1\delta_0(\rs_{\widetilde{m}}(n-s-1))\nonumber\\
&=&(-1)^{s}d^2\sum_{\substack{m'+a+b+2s+2\\+\sum_{i=1}^ka_i=m}}\left\{\prod_{j=0}^{s-1}\left(j+a+1+\sum_{i=1}^s (a_i+1)\right)\right\}\theta_{\vec{a},s}^0\theta_a^1\theta_b^1\rs_{m'}(n-s-2)\nonumber\\
&=&-d\sum_{b+1+m''=m}\theta_b^1\times\nonumber\\
&&\times(-1)^sd\left(\sum_{\substack{m'+a+1+2s\\+\sum_{i=1}^ka_i=m''}}\left\{\prod_{j=0}^{s-1}\left(j+a+1+\sum_{i=1}^s (a_i+1)\right)\right\}\theta_{\vec{a},s}^0\theta_a^1\rs_{m'}(n-s-2)
\right)\nonumber\\ &=&-d\sum_{b+1+m''=m}\theta_b^1\cdot(\delta_1)^s\delta_0(\rs_{m''}(n-1))\nonumber\\
&=&-d\sum_{b+1+m''=m}\theta_b^1\cdot[1](\delta_1)^s\delta_0(\rs_{m''}(n))\nonumber\\ \end{eqnarray}

By Lemma \ref{lbz} and (\ref{add1}) we have
\begin{eqnarray}&&(\delta_1)^{k-1-s}\delta_0(\delta_1)^s\delta_0(\rs_m(n))\nonumber\\
&=&-d\sum_{b+1+m''=m}(\delta_1)^{k-1-s}\left(\theta_b^1\cdot[1](\delta_1)^s\delta_0(\rs_{m'}(n))\right)\nonumber\\
&=&-d\sum_{b+1+m''=m}\sum_{i=0}^{k-s-1}\left\{\binom{k-s-1}{i}\cdot[k-i-s]((\delta_1)^{i+s}\delta_0(\rs_{m''}(n)))\times\right.\nonumber\\
&&\times\left.[i]((\delta_1)^{k-s-i-1}(\theta^1_b\rs_0(n-s-2)))\right\}\nonumber\\
&=&-d\sum_{b+1+m''=m}\sum_{i=0}^{k-s-1}\left\{\binom{k-s-1}{i}\cdot (-1)^{s+i+1}d~\times\right.\nonumber\\ &&\left(\sum_{\substack{m_1+a+2(s+i)+1\\+\sum_{l=1}^{s+i}a_l=m''}}\prod_{j=0}^{i+s-1}\left(j+a+1+\sum_{i=1}^{s+i} (a_i+1)\right)\theta_{\vec{a},i+s}^0\theta_a^1\rs_{m_1}(n-k-1)\right)\nonumber\\
&&\times \left.[i]((\delta_1)^{k-s-i-1}(\theta^1_b\rs_0(n-s-2)))\right\}\nonumber\\
&=&d^2\sum_{\substack{m_1+a+2(s+i)+1\\+b+1\\+\sum_{l=1}^{s+i}a_l=m}}\sum_{i=0}^{k-s-1}\left\{\binom{k-s-1}{i}\cdot [i]((\delta_1)^{k-s-i-1}(\theta^1_b\rs_0(n-s-2)))\times\right.\nonumber\\ &&\times\left.(-1)^{s+i}~\left(\prod_{j=0}^{i+s-1}\left(j+a+1+\sum_{i=1}^{s+i} (a_i+1)\right)\theta_{\vec{a},i+s}^0\theta_a^1\rs_{m_1}(n-k-1)\right)\right\}\nonumber\\
&=&d^2\sum_{\substack{m_1+a+1\\+2(s+i)+b+1\\+\sum_{l=1}^{s+i}a_l=m}}\sum_{i=0}^{k-s-1}\left\{\binom{k-s-1}{i}\right.\cdot[i]((\delta_1)^{k-s-i-1}(\theta^1_b\rs_{m_1}(n-s-2)))\times\nonumber\\&& (-1)^{s+i}\left.\left(\prod_{j=0}^{i+s-1}\left(j+a+1+\sum_{i=1}^{s+i} (a_i+1)\right)\theta_{\vec{a},i+s}^0\theta_a^1\right)\right\},\nonumber\\
\end{eqnarray}
where the last equality is because by Lemma \ref{lbz} we have
\begin{eqnarray}&&[i]((\delta_1)^{k-s-i-1}(\theta^1_b\rs_0(n-s-2)))\cdot\rs_{m_1}(n-k-1)\nonumber\\
&=&[i]((\delta_1)^{k-s-i-1}(\theta^1_b\rs_0(n-s-2))\cdot\rs_{m_1}(n-k-1+i))\nonumber\\
&=&[i]((\delta_1)^{k-s-i-1}(\theta^1_b\rs_{m_1}(n-s-2))).\end{eqnarray}
Let $m_1+b+1=m_2$, then we have
\begin{eqnarray}\label{del101}&&(\delta_1)^{k-1-s}\delta_0(\delta_1)^s\delta_0(\rs_m(n))\nonumber\\
&=&d^2\sum_{\substack{m_2+a+2(s+i)+1\\+\sum_{l=1}^{s+i}a_l=m}}\sum_{i=0}^{k-s-1}\left\{(-1)^{s+i}\binom{k-s-1}{i}\times\right.\nonumber\\
&&\times\left(\prod_{j=0}^{i+s-1}\left(j+a+1+\sum_{i=1}^{s+i} (a_i+1)\right)\theta_{\vec{a},i+s}^0\theta_a^1\right)\nonumber\\&&\times \left.[i]((\delta_1)^{k-s-i-1}\left(\sum_{b+m_1+1=m_2}\theta^1_b\rs_{m_1}(n-s-2))\right)\right\}\nonumber\\
&=&d\sum_{\substack{m_2+a+2(s+i)+1\\+\sum_{l=1}^{s+i}a_l=m}}\sum_{i=0}^{k-s-1}\left\{\binom{k-s-1}{i}\cdot\right.[i]((\delta_1)^{k-s-i-1}\delta_0(\rs_{m_2}(n-s-1))\nonumber\\
&&\times\left.(-1)^{s+i+1}\left(\prod_{j=0}^{i+s-1}\left(j+a+1+\sum_{i=1}^{s+i} (a_i+1)\right)\theta_{\vec{a},i+s}^0\theta_a^1\right)\right\}\nonumber\\
&=&(-1)^{k+1}d^2\sum_{\substack{m_3+a+b+2\\+2(k-1)+\sum_{l=1}^{s+i}a_l\\+\sum_{l=1}^{k-s-i-1}b_l=m}}\sum_{i=0}^{k-s-1}\left\{\binom{k-s-1}{i}\times\right.\nonumber\\
&&\times\left(\prod_{j=0}^{k-s-i-2}\left(j+b+1+\sum_{i=1}^{k-s-i-1} (b_i+1)\right)\theta_{\vec{b},k-i-s-1}^0\theta_b^1\right)\nonumber\\
&&\times\left.\left(\prod_{j=0}^{i+s-1}\left(j+a+1+\sum_{i=1}^{s+i} (a_i+1)\right)\theta_{\vec{a},i+s}^0\theta_a^1\right)\right\}\rs_{m_3}(n-k-1)\nonumber
\end{eqnarray}
Hence (\ref{del101s}) is proved.  (\ref{del101a}) follows directly from (\ref{del101s}) and Lemma \ref{combi} below.  

The lemma is proved.
\end{proof}
\begin{lemma}\label{combi}For any $0\leq M\leq k-1$, we have
\[\sum_{\substack{s,i\geq 0\\ s+i=M}}\binom{k-s-1}{i}=\binom{k}{M}.\]
\end{lemma}
\begin{proof}We have
\begin{eqnarray}\sum_{\substack{s,i\geq 0\\ s+i=M}}\binom{k-s-1}{i}&=&\text{Coeff}_{x^{k-M-1}}\left(\sum_{s=0}^M(x+1)^{k-s-1}\right)\nonumber\\
&=&\text{Coeff}_{x^{k-M-1}}\left((x+1)^{k-M-1}\frac{(x+1)^{M+1}-1}{x+1-1}\right)\nonumber\\
&=&\text{Coeff}_{x^{k-M-1}}\left(\frac{(x+1)^{k}-(x+1)^{k-M-1}}{x}\right)\nonumber\\
&=&\text{Coeff}_{x^{k-M}}\left((x+1)^{k}-(x+1)^{k-M-1}\right)\nonumber\\
&=&\binom{k}{M}.\nonumber
\end{eqnarray}
\end{proof}
\begin{lemma}\label{combi2}For any $a,N,m,k\in\bz_{\geq 0}$, we always have
\begin{eqnarray}\label{anmk}&&\sum_{M=0}^{k}N\binom{k+1}{M}\sum_{\sum_{i=1}^{k+1}a_i=m}\theta^0_{\vec{a},k+1}\cdot\left\{\left(\prod_{j=0}^{M-1}(j+a+1+\sum_{i=1}^M(a_i+1))\right)\times\right.\nonumber\\
&&\times \left.\left(\prod_{j=0}^{k-M-1}(j+a_{k+1}+1+\sum_{i=M+1}^k(a_i+1)+N+1)\right)\right\}\nonumber\\
&=&\sum_{\sum_{i=1}^{k+1}a_i=m}\theta^0_{\vec{a},k+1}\cdot\left(\prod_{j=0}^{k}(j+a+1+\sum_{i=1}^{k+1}(a_i+1)+N)\right)\nonumber\\&&-\sum_{\sum_{i=1}^{k+1}a_i=m}\theta^0_{\vec{a},k+1}\cdot\left(\prod_{j=0}^{k}(j+a+1+\sum_{i=1}^{k+1}(a_i+1))\right)\nonumber\\
\end{eqnarray}
\end{lemma}
\begin{proof}We do the induction on both $N$ and $k$.  Obviously, the lemma holds for all $k\geq 0$ if $N=0$ and for all $N\geq 0$ if $k=0$.  Assume the lemma holds for all $N\geq 0$ if $k\leq k_0$ and for all $N\leq N_0$ if $k=k_0+1$.  It suffices to show that the lemma holds for $k=k_0+1$ and $N=N_0+1$. 

We have for any $N,k\geq 0,P\in\bz$
\begin{eqnarray}\label{nkp}&&\prod_{j=0}^{k}(j+P+N+1)=\prod_{j=1}^{k+1}(j+P+N)\nonumber\\&=&\prod_{j=0}^{k}(j+P+N)+(k+1)\prod_{j=0}^{k-1}(j+P+N+1)
\end{eqnarray}

Denote by $L(m,k,N)$ ( $R(m,k,N)$, resp.) the left (right, resp.) hand side of (\ref{anmk}).  By (\ref{nkp}) we have
\begin{eqnarray}
&&R(m,k_0+1,N_0+1)=\nonumber\\
&&\sum_{\sum_{i=1}^{k_0+2}a_i=m}\theta^0_{\vec{a},k_0+2}\cdot\left(\prod_{j=0}^{k_0+1}(j+a+1+\sum_{i=1}^{k_0+2}(a_i+1)+N_0+1)\right)\nonumber\\&&-\sum_{\sum_{i=1}^{k_0+2}a_i=m}\theta^0_{\vec{a},k_0+2}\cdot\left(\prod_{j=0}^{k_0+1}(j+a+1+\sum_{i=1}^{k_0+2}(a_i+1))\right)\nonumber\\
&=&\sum_{\sum_{i=1}^{k_0+2}a_i=m}\theta^0_{\vec{a},k_0+2}\cdot\left(\prod_{j=0}^{k_0+1}(j+a+1+\sum_{i=1}^{k_0+2}(a_i+1)+N_0)\right)\nonumber\\&&-\sum_{\sum_{i=1}^{k_0+2}a_i=m}\theta^0_{\vec{a},k_0+2}\cdot\left(\prod_{j=0}^{k_0+1}(j+a+1+\sum_{i=1}^{k_0+2}(a_i+1))\right)\nonumber\\
&&+\sum_{\sum_{i=1}^{k_0+2}a_i=m}(k_0+2)\theta^0_{\vec{a},k_0+2}\cdot\left(\prod_{j=0}^{k_0}(j+a+1+\sum_{i=1}^{k_0+2}(a_i+1)+N_0+1)\right)\nonumber\\
\end{eqnarray}
Therefore
\begin{eqnarray}\label{rmkn}&&R(m,k_0+1,N_0+1)-R(m,k_0+1,N_0)\nonumber\\
&=&\sum_{\sum_{i=1}^{k_0+2}a_i=m}(k_0+2)\theta^0_{\vec{a},k_0+2}\cdot\left(\prod_{j=0}^{k_0}(j+a+1+\sum_{i=1}^{k_0+2}(a_i+1)+N_0+1)\right)\nonumber\\
&=&(k_0+2)\sum_{a_{k_0+2}=0}^m\theta_{a_{k_0+2}}^0\cdot R(m-a_{k_0+2},k_0,N_0+a_{k_0+2}+2)\nonumber\\
&&+(k_0+2)\sum_{\sum_{i=1}^{k_0+2}a_i=m}\theta^0_{\vec{a},k_0+2}\cdot\left(\prod_{j=0}^{k_0}(j+a+1+\sum_{i=1}^{k_0+1}(a_i+1))\right)
\end{eqnarray}
On the other hand by (\ref{nkp}) we have\footnotesize
\begin{eqnarray}\label{lmkn1}&&L(m,k_0+1,N_0+1)\nonumber\\
&=&\sum_{M=0}^{k_0+1}(N_0+1)\binom{k_0+2}{M}\sum_{\sum_{i=1}^{k_0+2}a_i=m}\theta^0_{\vec{a},k_0+2}\cdot\left\{\left(\prod_{j=0}^{M-1}(j+a+1+\sum_{i=1}^M(a_i+1))\right)\times\right.\nonumber\\
&&\times \left.\left(\prod_{j=0}^{k_0-M}(j+a_{k_0+2}+1+\sum_{i=M+1}^{k_0+1}(a_i+1)+N_0+2)\right)\right\}\nonumber\\
&=&\sum_{M=0}^{k_0+1}(N_0+1)\binom{k_0+2}{M}\sum_{\sum_{i=1}^{k_0+2}a_i=m}\theta^0_{\vec{a},k_0+2}\cdot\left\{\left(\prod_{j=0}^{M-1}(j+a+1+\sum_{i=1}^M(a_i+1))\right)\times\right.\nonumber\\
&&\times \left.\left(\prod_{j=0}^{k_0-M}(j+a_{k_0+2}+1+\sum_{i=M+1}^{k_0+1}(a_i+1)+N_0+1)\right)\right\}\nonumber\\
&&+\sum_{M=0}^{k_0}(N_0+1)(k_0-M+1)\binom{k_0+2}{M}\sum_{\sum_{i=1}^{k_0+2}a_i=m}\theta^0_{\vec{a},k_0+2}\cdot\left\{\left(\prod_{j=0}^{M-1}(j+a+1+\sum_{i=1}^M(a_i+1))\right)\times\right.\nonumber\\
&&\times \left.\left(\prod_{j=0}^{k_0-M-1}(j+a_{k_0+2}+1+\sum_{i=M+1}^{k_0+1}(a_i+1)+N_0+2)\right)\right\}\nonumber\\
&=&L(m,k_0+1,N_0)+\sum_{M=0}^{k_0+1}\binom{k_0+2}{M}\sum_{\sum_{i=1}^{k_0+2}a_i=m}\theta^0_{\vec{a},k_0+2}\cdot\left\{\left(\prod_{j=0}^{M-1}(j+a+1+\sum_{i=1}^M(a_i+1))\right)\times\right.\nonumber\\
&&\times \left.\left(\prod_{j=0}^{k_0-M}(j+a_{k_0+2}+1+\sum_{i=M+1}^{k_0+1}(a_i+1)+N_0+1)\right)\right\}\nonumber\\
&&+\sum_{M=0}^{k_0}(N_0+1)(k_0-M+1)\binom{k_0+2}{M}\sum_{\sum_{i=1}^{k_0+2}a_i=m}\theta^0_{\vec{a},k_0+2}\cdot\left\{\left(\prod_{j=0}^{M-1}(j+a+1+\sum_{i=1}^M(a_i+1))\right)\times\right.\nonumber\\
&&\times \left.\left(\prod_{j=0}^{k_0-M-1}(j+a_{k_0+2}+1+\sum_{i=M+1}^{k_0+1}(a_i+1)+N_0+2)\right)\right\}\nonumber\\
&=&L(m,k_0+1,N_0)+\binom{k_0+2}{k_0+1}\sum_{\sum_{i=1}^{k_0+2}a_i=m}\theta^0_{\vec{a},k_0+2}\cdot\left(\prod_{j=0}^{k_0}(j+a+1+\sum_{i=1}^{k_0+1}(a_i+1))\right)\nonumber\\
&&+\sum_{M=0}^{k_0}\binom{k_0+2}{M}\sum_{\sum_{i=1}^{k_0+2}a_i=m}\theta^0_{\vec{a},k_0+2}\cdot\left\{\left(\prod_{j=0}^{M-1}(j+a+1+\sum_{i=1}^M(a_i+1))\right)\times\right.\nonumber\\
&&\times \left.\left(\prod_{j=0}^{k_0-M-1}(j+a_{k_0+2}+1+\sum_{i=M+1}^{k_0+1}(a_i+1)+N_0+2)\right)\right\}\left(\sum_{i=M+1}^{k_0+2}(a_i+1)\right)\nonumber\\
&&+\sum_{M=0}^{k_0}(N_0+1)(k_0-M+2)\binom{k_0+2}{M}\sum_{\sum_{i=1}^{k_0+2}a_i=m}\theta^0_{\vec{a},k_0+2}\cdot\left\{\left(\prod_{j=0}^{M-1}(j+a+1+\sum_{i=1}^M(a_i+1))\right)\times\right.\nonumber\\
&&\times \left.\left(\prod_{j=0}^{k_0-M-1}(j+a_{k_0+2}+1+\sum_{i=M+1}^{k_0+1}(a_i+1)+N_0+2)\right)\right\}\nonumber\\
\end{eqnarray}
\normalsize
For each $0\leq M\leq k_0$, by switching $a_{k_0+2}$ and $a_{l}$ we have for any $M+1\leq l\leq k_0+1$ 
\begin{eqnarray}&&\sum_{\sum_{i=1}^{k_0+2}a_i=m}\theta^0_{\vec{a},k_0+2}(a_{l}+1)\cdot\left\{\left(\prod_{j=0}^{M-1}(j+a+1+\sum_{i=1}^M(a_i+1))\right)\times\right.\nonumber\\
&&\times \left.\left(\prod_{j=0}^{k_0-M-1}(j+a_{k_0+2}+1+\sum_{i=M+1}^{k_0+1}(a_i+1)+N_0+2)\right)\right\}\nonumber\\
&=&\sum_{\sum_{i=1}^{k_0+2}a_i=m}\theta^0_{\vec{a},k_0+2}(a_{k_0+2}+1)\cdot\left\{\left(\prod_{j=0}^{M-1}(j+a+1+\sum_{i=1}^M(a_i+1))\right)\times\right.\nonumber\\
&&\times \left.\left(\prod_{j=0}^{k_0-M-1}(j+a_{k_0+2}+1+\sum_{i=M+1}^{k_0+1}(a_i+1)+N_0+2)\right)\right\}\nonumber
\end{eqnarray}
Therefore
\begin{eqnarray}\label{alak}
&&\sum_{M=0}^{k_0}\binom{k_0+2}{M}\sum_{\sum_{i=1}^{k_0+2}a_i=m}\theta^0_{\vec{a},k_0+2}\cdot\left\{\left(\prod_{j=0}^{M-1}(j+a+1+\sum_{i=1}^M(a_i+1))\right)\times\right.\nonumber\\
&&\times \left.\left(\prod_{j=0}^{k_0-M-1}(j+a_{k_0+2}+1+\sum_{i=M+1}^{k_0+1}(a_i+1)+N_0+2)\right)\right\}\left(\sum_{l=M+1}^{k_0+2}(a_l+1)\right)\nonumber\\
&=&\sum_{M=0}^{k_0}(k_0-M+2)\binom{k_0+2}{M}\sum_{\sum_{i=1}^{k_0+2}a_i=m}\theta^0_{\vec{a},k_0+2}\cdot\left\{\left(\prod_{j=0}^{M-1}(j+a+1+\sum_{i=1}^M(a_i+1))\right)\times\right.\nonumber\\
&&\times \left.\left(\prod_{j=0}^{k_0-M-1}(j+a_{k_0+2}+1+\sum_{i=M+1}^{k_0+1}(a_i+1)+N_0+2)\right)\right\}(a_{k_0+2}+1)\nonumber\\
&=&\sum_{M=0}^{k_0}(k_0+2)\binom{k_0+1}{M}\sum_{\sum_{i=1}^{k_0+2}a_i=m}\theta^0_{\vec{a},k_0+2}\cdot\left\{\left(\prod_{j=0}^{M-1}(j+a+1+\sum_{i=1}^M(a_i+1))\right)\times\right.\nonumber\\
&&\times \left.\left(\prod_{j=0}^{k_0-M-1}(j+a_{k_0+2}+1+\sum_{i=M+1}^{k_0+1}(a_i+1)+N_0+2)\right)\right\}(a_{k_0+2}+1)
\nonumber\\
\end{eqnarray}
Hence by (\ref{alak}) and (\ref{lmkn1}) we have
\begin{eqnarray}\label{lmkn2}&&L(m,k_0+1,N_0+1)-L(m,k_0+1,N_0)\nonumber\\
&=&(k_0+2)\sum_{M=0}^{k_0+1}\binom{k_0+1}{M}\sum_{\sum_{i=1}^{k_0+2}a_i=m}\theta^0_{\vec{a},k_0+2}\cdot\left\{\left(\prod_{j=0}^{M-1}(j+a+1+\sum_{i=1}^M(a_i+1))\right)\times\right.\nonumber\\
&&\times \left.\left(\prod_{j=0}^{k_0-M-1}(j+a_{k_0+2}+1+\sum_{i=M+1}^{k_0+1}(a_i+1)+N_0+2)\right)\right\}(N_0+1+a_{k_0+2}+1)\nonumber\\
&&+(k_0+2)\sum_{\sum_{i=1}^{k_0+2}a_i=m}\theta^0_{\vec{a},k_0+2}\cdot\left(\prod_{j=0}^{k_0}(j+a+1+\sum_{i=1}^{k_0+1}(a_i+1))\right)\nonumber\\
&=&(k_0+2)\sum_{a_{k_0+2}=0}^m\theta_{a_{k_0+2}}^0\cdot L(m-a_{k_0+2},k_0,N_0+a_{k_0+2}+2)\nonumber\\
&&+(k_0+2)\sum_{\sum_{i=1}^{k_0+2}a_i=m}\theta^0_{\vec{a},k_0+2}\cdot\left(\prod_{j=0}^{k_0}(j+a+1+\sum_{i=1}^{k_0+1}(a_i+1))\right)\nonumber\\
\end{eqnarray}
By (\ref{rmkn}), (\ref{lmkn2}) and induction assumption we have
\[R(m,k_0+1,N_0+1)=L(m,k_0+1,N_0+1).\] 
The lemma is proved.
\end{proof}
\begin{rem}Since both sides of (\ref{anmk}) are polynomials in $N$, Lemma \ref{combi2} still holds for all $a,m,k\in\bz_{\geq0}$ and $N\in\bc$, and also with $\theta_{a}^0$ replaced by any fixed function in $a$.\end{rem}
\begin{proof}[Proof of Proposition \ref{wkmain}]By (\ref{del2}) and Lemma \ref{del21S} we have\footnotesize
\begin{eqnarray}&&(-1)^{k+1}\delta_2(\delta_1)^k\delta_0(\rs_m(n))\nonumber\\
&=&d^2\sum_{\substack{u\geq 0,t\geq0\\m= m'+a\\+2k+t+2\\+\sum_{i=1}^ka_i}}\left\{\prod_{j=0}^{k-1}\left(j+a+1+\sum_{i=1}^k (a_i+1)\right)\right\}\theta_{\vec{a},k}^0\cdot\theta_{a-u-2}^0\theta_{u+t}^0(u+1)\rs_{m'}(n-k-2)\nonumber\\
&-3d&\sum_{\substack{u\geq 0,\\m= m'+a\\+2k+1\\+\sum_{i=1}^ka_i}}\left\{\prod_{j=0}^{k-1}\left(j+a+1+\sum_{i=1}^k (a_i+1)\right)\right\}\theta_{\vec{a},k}^0\cdot\theta_{a-u-3}^0\theta_{u}^0\frac{(u+1)(u+2)}2\rs_{m'}(n-k-2)\nonumber
\end{eqnarray}\normalsize
By Lemma \ref{del31S} and (\ref{del3}) we have 
\begin{eqnarray}&&(-1)^{k+1}\sum_{s=0}^{k-1}\delta_3(\delta_1)^{k-1-s}\delta_0(\delta_1)^s\delta_0(\rs_m(n))\nonumber\\&=&d^2\sum_{M=0}^{k-1}\sum_{\substack{u\geq0,t\geq0\\m'+a+b+2\\+2(k-1)+\sum_{l=1}^{M}a_l\\+\sum_{l=1}^{k-M-1}b_l=m}}\left\{\binom{k}{M}\cdot\theta_{u+t}^0(u+1)(t+1)\times\right.\nonumber\\
&&\times\left(\prod_{j=0}^{k-M-2}\left(j+b+1+\sum_{i=1}^{k-M-1} (b_i+1)\right)\theta_{\vec{b},k-M-1}^0\theta_{b-t-2}^0\right)\nonumber\\
&&\times\left.\left(\prod_{j=0}^{M-1}\left(j+a+1+\sum_{i=1}^{M} (a_i+1)\right)\theta_{\vec{a},M}^0\theta_{a-u-2}^0\right)\right\}\rs_{m'}(n-k-2).\nonumber
\end{eqnarray}
Let $a_{k}=b-t-2$, $a_{M+i}=b_i$ for $1\leq i\leq k-M-1$.  Then we have
\begin{eqnarray}&&(-1)^{k+1}\sum_{s=0}^{k-1}\delta_3(\delta_1)^{k-1-s}\delta_0(\delta_1)^s\delta_0(\rs_m(n))\nonumber\\&=&d^2\sum_{M=0}^{k-1}\sum_{\substack{u\geq0,t\geq 0\\m=m'+a\\+2k+t+2\\ +\sum_{l=1}^{k}a_l}}\left\{\binom{k}{M}\cdot\theta_{u+t}^0(u+1)\times\right.\nonumber\\
&&(t+1)\times\prod_{j=0}^{k-M-2}\left(j+a_k+1+\sum_{i=1}^{k-M-1} (b_i+1)+t+2\right)\nonumber\\
&&\times\left.\prod_{j=0}^{M-1}\left(j+a+1+\sum_{i=1}^{M} (a_i+1)\right)\right\}\theta_{\vec{a},k}^0\theta_{a-u-2}^0\rs_{m'}(n-k-2).\nonumber\\
&=&d^2\sum_{\substack{u\geq0,t\geq 0\\m=m'+a\\+2k+t+2\\ +\sum_{l=1}^{k}a_l}}\sum_{M=0}^{k-1}\left\{\binom{k}{M}\cdot\theta_{u+t}^0(u+1)\times\right.\nonumber\\
&&(t+1)\times\prod_{j=0}^{k-M-2}\left(j+a_k+1+\sum_{i=1}^{k-M-1} (b_i+1)+t+2\right)\nonumber\\
&&\times\left.\prod_{j=0}^{M-1}\left(j+a+1+\sum_{i=1}^{M} (a_i+1)\right)\right\}\theta_{\vec{a},k}^0\theta_{a-u-2}^0\rs_{m'}(n-k-2).\nonumber
\end{eqnarray}
Therefore by Lemma \ref{combi2} we have\footnotesize
\begin{eqnarray}\label{s2+s3}&&(-1)^{k+1}\left(\delta_2(\delta_1)^k\delta_0+\sum_{s=0}^{k-1}\delta_3(\delta_1)^{k-1-s}\delta_0(\delta_1)^s\delta_0\right)(\rs_m(n))\nonumber\\&=&d^2\sum_{\substack{u\geq 0,t\geq0\\m= m'+a\\+2k+t+2\\+\sum_{i=1}^ka_i}}\left\{\prod_{j=0}^{k-1}\left(j+a+t+2+\sum_{i=1}^k (a_i+1)\right)\right\}\theta_{\vec{a},k}^0\cdot\theta_{a-u-2}^0\theta_{u+t}^0(u+1)\rs_{m'}(n-k-2)\nonumber\\
&-3d&\sum_{\substack{u\geq 0,\\m= m'+a\\+2k+1\\+\sum_{i=1}^ka_i}}\left\{\prod_{j=0}^{k-1}\left(j+a+1+\sum_{i=1}^k (a_i+1)\right)\right\}\theta_{\vec{a},k}^0\cdot\theta_{a-u-3}^0\theta_{u}^0\frac{(u+1)(u+2)}2\rs_{m'}(n-k-2)\nonumber\\
\end{eqnarray}\normalsize
Let $a'=a+t+1$ and $u'=u+t$, then we have\footnotesize
\begin{eqnarray}\label{s2+s3+}&&\sum_{\substack{u\geq 0,t\geq0\\m= m'+a\\+2k+t+2\\+\sum_{i=1}^ka_i}}\left\{\prod_{j=0}^{k-1}\left(j+a+t+2+\sum_{i=1}^k (a_i+1)\right)\right\}\theta_{\vec{a},k}^0\cdot\theta_{a-u-2}^0\theta_{u+t}^0(u+1)\rs_{m'}(n-k-2)\nonumber\\
&=&\sum_{\substack{u'\geq 0,\\m= m'+a'\\+2k+1\\+\sum_{i=1}^ka_i}}\left\{\prod_{j=0}^{k-1}\left(j+a'+1+\sum_{i=1}^k (a_i+1)\right)\right\}\theta_{\vec{a},k}^0\cdot\theta_{a-u-2}^0\theta_{u'}^0\left(\sum_{u=0}^{u'}(u+1)\right)\rs_{m'}(n-k-2)\nonumber\\
&=&\sum_{\substack{u'\geq 0,\\m= m'+a'\\+2k+1\\+\sum_{i=1}^ka_i}}\left\{\prod_{j=0}^{k-1}\left(j+a'+1+\sum_{i=1}^k (a_i+1)\right)\right\}\theta_{\vec{a},k}^0\cdot\theta_{a-u-2}^0\theta_{u'}^0\frac{(u'+1)(u'+2)}2\rs_{m'}(n-k-2)\nonumber\\
\end{eqnarray}\normalsize
Combine (\ref{s2+s3}) and (\ref{s2+s3+}) and we have 
\[\left(\delta_2(\delta_1)^k\delta_0+\sum_{s=0}^{k-1}\delta_3(\delta_1)^{k-1-s}\delta_0(\delta_1)^s\delta_0\right)(\rs_m(n))\in d(d-3)A^*(S^{[n-k-2]})\]
\end{proof}
\begin{rem}\label{k=0}By our convention $\sum_{s=0}^{k-1}\delta_3(\delta_1)^{k-1-s}\delta_0(\delta_1)^s\delta_0=0$ for $k=0$, hence by Lemma \ref{comm02} and Remark \ref{lag} we have 
\[\delta_2\delta_0(\rs_m(n))=\delta_2\delta_0(\rs_m(n))-\delta_0\delta_2\rs_m(n)\in (d-3)A^*(S^{[n-2]}).\]
Therefore Proposition \ref{wkmain} also holds for $k=0$.
\end{rem}

\subsection{Some properties on commutators.}Let $\xi:\bigoplus_{n\geq 0} B(n)\ra\bigoplus_{n\geq 0} B(n)$ be any group homomorphism.  Define $Ad_{\delta_i}(\xi):=[\delta_i,\xi]=\delta_i\xi-\xi\delta_i$. 
\begin{prop}\label{vkmain}For any $k\geq0$ and $0\leq i\leq k$, we have
\begin{eqnarray}&&\left\{(Ad_{\delta_1})^i(\delta_2)\cdot(\delta_1)^{k-i}\delta_0+\sum_{s=0}^{i-1}(Ad_{\delta_1})^{i-1-s}Ad_{\delta_0}(Ad_{\delta_1})^s(\delta_3)\cdot(\delta_1)^{k-i}\delta_0\right.\nonumber\\
&+&\left.(Ad_{\delta_1})^i(\delta_3)\cdot(\sum_{s=0}^{k-i-1}(\delta_1)^{k-i-1-s}\delta_0(\delta_1)^s\delta_0)\right\}(\rs_m(n))\in(d-3)A^*(S^{[n-k-2]})\nonumber\end{eqnarray}
In particular
\begin{eqnarray}\label{ckcom}&& \left\{Ad_{\delta_0}(Ad_{\delta_1})^k(\delta_2)+\sum_{s=0}^{k-1}Ad_{\delta_0}(Ad_{\delta_1})^{k-1-s}Ad_{\delta_0}(Ad_{\delta_1})^s(\delta_3)\right\}(\rs_m(n))\nonumber\\
&=&\left\{(Ad_{\delta_1})^k(\delta_2)+\sum_{s=0}^{k-1}(Ad_{\delta_1})^{k-1-s}Ad_{\delta_0}(Ad_{\delta_1})^s(\delta_3)\right\}\delta_0(\rs_m(n))
\in(d-3)A^*(S^{[n-k-2]}).\nonumber\\\end{eqnarray}
\end{prop}
\begin{proof}We do the induction on both $i$ and $k$.  The proposition holds for any $k\geq 0$ if $i=0$ by Proposition \ref{wkmain} and for all $0\leq i\leq k$ if $k=0$.  Assume the proposition holds for all $0\leq i\leq k_0$, $k\leq k_0$ and for all $0\leq i\leq i_0\leq k_0$, $k=k_0+1$.  It suffices to show that the proposition holds for $k=k_0+1$ and $i=i_0+1$.

We have
\begin{eqnarray}&&(Ad_{\delta_1})^{i_0+1}(\delta_2)\cdot(\delta_1)^{k_0-i_0}\delta_0+\sum_{s=0}^{i_0}(Ad_{\delta_1})^{i_0-s}Ad_{\delta_0}(Ad_{\delta_1})^s(\delta_3)\cdot(\delta_1)^{k_0-i_0}\delta_0\nonumber\\
&&+(Ad_{\delta_1})^{i_0+1}(\delta_3)\cdot(\sum_{s=0}^{k_0-i_0-1}(\delta_1)^{k_0-i_0-1-s}\delta_0(\delta_1)^s\delta_0)\nonumber\\
&=&\delta_1\left\{(Ad_{\delta_1})^{i_0}(\delta_2)\cdot(\delta_1)^{k_0-i_0}\delta_0+\sum_{s=0}^{i_0-1}(Ad_{\delta_1})^{i_0-1-s}Ad_{\delta_0}(Ad_{\delta_1})^s(\delta_3)\cdot(\delta_1)^{k_0-i_0}\delta_0\right.\nonumber\\
&&+\left.(Ad_{\delta_1})^{i_0}(\delta_3)\cdot(\sum_{s=0}^{k_0-i_0-1}(\delta_1)^{k_0-i_0-1-s}\delta_0(\delta_1)^s\delta_0)\right\}+\delta_0\cdot((Ad_{\delta_1})^{i_0}(\delta_3))\cdot(\delta_1)^{k_0-i_0}\delta_0\nonumber\\
&&-\left\{(Ad_{\delta_1})^{i_0}(\delta_2)\cdot(\delta_1)^{k_0-i_0+1}\delta_0+\sum_{s=0}^{i_0-1}(Ad_{\delta_1})^{i_0-1-s}Ad_{\delta_0}(Ad_{\delta_1})^s(\delta_3)\cdot(\delta_1)^{k_0-i_0+1}\delta_0\right.\nonumber\\
&&+\left.(Ad_{\delta_1})^{i_0}(\delta_3)\cdot(\sum_{s=0}^{k_0-i_0-1}(\delta_1)^{k_0-i_0-s}\delta_0(\delta_1)^s\delta_0)+(Ad_{\delta_1})^{i_0}(\delta_3)\cdot\delta_0(\delta_1)^{k_0-i_0}\delta_0\right\}
\nonumber\end{eqnarray}
By induction assumption we have
\begin{eqnarray}&&\left\{(Ad_{\delta_1})^{i_0}(\delta_2)\cdot(\delta_1)^{k_0-i_0}\delta_0+\sum_{s=0}^{i_0-1}(Ad_{\delta_1})^{i_0-1-s}Ad_{\delta_0}(Ad_{\delta_1})^s(\delta_3)\cdot(\delta_1)^{k_0-i_0}\delta_0\right.\nonumber\\
&&+\left.(Ad_{\delta_1})^{i_0}(\delta_3)\cdot(\sum_{s=0}^{k_0-i_0-1}(\delta_1)^{k_0-i_0-1-s}\delta_0(\delta_1)^s\delta_0)\right\}(\rs_m(n))\in (d-3)A^*(S^{[n-k_0-2]});\nonumber
\end{eqnarray}
and
\begin{eqnarray}&&\left\{(Ad_{\delta_1})^{i_0}(\delta_2)\cdot(\delta_1)^{k_0-i_0+1}\delta_0+\sum_{s=0}^{i_0-1}(Ad_{\delta_1})^{i_0-1-s}Ad_{\delta_0}(Ad_{\delta_1})^s(\delta_3)\cdot(\delta_1)^{k_0-i_0+1}\delta_0\right.\nonumber\\
&&+\left.(Ad_{\delta_1})^{i_0}(\delta_3)\cdot(\sum_{s=0}^{k_0-i_0-1}(\delta_1)^{k_0-i_0-s}\delta_0(\delta_1)^s\delta_0)+(Ad_{\delta_1})^{i_0}(\delta_3)\cdot\delta_0(\delta_1)^{k_0-i_0}\delta_0\right\}(\rs_m(n))\nonumber\\
&=&\left\{(Ad_{\delta_1})^{i_0}(\delta_2)\cdot(\delta_1)^{k_0-i_0+1}\delta_0+\sum_{s=0}^{i_0-1}(Ad_{\delta_1})^{i_0-1-s}Ad_{\delta_0}(Ad_{\delta_1})^s(\delta_3)\cdot(\delta_1)^{k_0-i_0+1}\delta_0\right.\nonumber\\
&&+\left.(Ad_{\delta_1})^{i_0}(\delta_3)\cdot(\sum_{s=0}^{k_0-i_0}(\delta_1)^{k_0-i_0-s}\delta_0(\delta_1)^s\delta_0)\right\}(\rs_m(n))\in (d-3)A^*(S^{[n-k_0-3]}).
\nonumber\end{eqnarray}
By Remark \ref{lag} we have
\[\delta_0\cdot((Ad_{\delta_1})^{i_0}(\delta_3))\cdot(\delta_1)^{k_0-i_0}\delta_0\rs_m(n)=0;\]
\[\delta_0\cdot\left\{(Ad_{\delta_1})^k(\delta_2)+\sum_{s=0}^{k-1}(Ad_{\delta_1})^{k-1-s}Ad_{\delta_0}(Ad_{\delta_1})^s(\delta_3)\right\}(\rs_m(n))=0.\]
Hence the proposition.
\end{proof}

Define 
\begin{eqnarray}&&(1):\bigoplus_{n\geq 0}B(n)\ra \bigoplus_{n\geq 0}B(n),\nonumber\\ &&(1)(\theta_{l_1}^0(n)\cdots\theta_{l_i}^0(n)\theta_{l_{i+1}}^1(n)\cdots\theta_{l_w}^1(n)\rs_m(n))\nonumber\\
&:=& \theta_{l_1}^0(n)\cdots\theta_{l_i}^0(n)\theta_{l_{i+1}}^1(n)\cdots\theta_{l_w}^1(n-k)\rs_{m-1}(n);\nonumber\end{eqnarray}
\[\mb:=\left\{\xi:\bigoplus_{n\geq 0}B(n)\ra \bigoplus_{n\geq 0}B(n)\left|\begin{array}{c}\xi\text{ is a group homomorphism}\\ \text{such that } \xi(1)=(1)\xi \text{ and }\xi[1]=[1]\xi.
\end{array}\right.\right\}\]

\begin{defn}\emph{A $0$-operator of degree $k~(k\geq 0)$} is an element $\xi\in\mb$ such that $\xi(B(n))\subset B(n-k)$ and 
 \[\xi(\theta_{l_1}^0\cdots\theta_{l_i}^0\theta_{l_{i+1}}^1\cdots\theta_{l_w}^1\rs_m(n))
=[k](\theta_{l_1}^0\cdots\theta_{l_i}^0\theta_{l_{i+1}}^1\cdots\theta_{l_w}^1)\cdot\xi(\rs_m(n)).\]
\emph{A $(1,0)$-operator of degree $k~(k\geq 0)$}
is an element $\xi\in\mb$ such that $\xi(B(n)_i^w)\subset B(n-k)^{w+k}_{i+k+1}$ and 
 \begin{eqnarray}&&\xi(\theta_{l_1}^0\cdots\theta_{l_i}^0\theta_{l_{i+1}}^1\cdots\theta_{l_w}^1\rs_m(n))\nonumber\\
&=&\sum_{j=i+1}^w[k](\theta_{l_1}^0\cdots\theta_{l_i}^0\theta_{l_{i+1}}^1\cdots\widehat{\theta^1_{l_j}}\cdots\theta_{l_w}^1)\cdot\xi(\theta_{l_j}^1\rs_m(n)).\nonumber\end{eqnarray}
\emph{A $(2,0)$-operator of degree $k~(k\geq 0)$}
is an element $\xi\in\mb$ such that $\xi(B(n)_i^w)\subset B(n-k)^{w+k}_{i+k+2}$ and
 \begin{eqnarray}&&\xi(\theta_{l_1}^0\cdots\theta_{l_i}^0\theta_{l_{i+1}}^1\cdots\theta_{l_w}^1\rs_m(n))\nonumber\\
&=&\sum_{i+1\leq j<t\leq w}[k](\theta_{l_1}^0\cdots\theta_{l_i}^0\theta_{l_{i+1}}^1\cdots\widehat{\theta^1_{l_j}}\cdots\widehat{\theta^1_{l_t}}\cdots\theta_{l_w}^1)\cdot\xi(\theta_{l_j}^1\theta_{l_t}^1\rs_m(n)).\nonumber\end{eqnarray}
\end{defn}
\begin{example}\begin{enumerate}\item $[k]$ is a 0-operator of degree $k$.
\item $\delta_0$ is a 0-operator of degree 1, $\delta_2$ is a $(1,0)$-operator of degree 1, and $\delta_3$ is a $(2,0)$-operator of degree 1.
\item $\delta_1\in\mb$ but it is not a $(1,0)$-operator.
\end{enumerate}
\end{example}
\begin{lemma}\label{adop}(1) If $\xi$ is a $(i,0)$-operator of degree $k$ for $i=1,2$, then $Ad_{\delta_1}(\xi)$ is a $(i,0)$-operator of degree $k+1$.

(2) If $\xi$ is a $(2,0)$-operator of degree $k$, then $Ad_{\delta_0}(\xi)$ is a $(1,0)$-operator of degree $k+1$.

(3) If $\xi$ is a $(1,0)$-operator of degree $k$, then $Ad_{\delta_0}(\xi)$ is a $0$-operator of degree $k+1$.
\end{lemma}
\begin{proof}Let $b=\theta_{l_1}^0\cdots\theta_{l_i}^0\theta_{l_{i+1}}^1\cdots\theta_{l_w}^1\rs_m(n+1)$.  Let $\xi$ be a $(1,0)$-operator of degree $k$.  Obviously $Ad_{\delta_1}(\xi)\in\mb$ and $Ad_{\delta_1}(\xi)(B(n)_i^w)\subset B(n-k-1)_{i+k+2}^{w+k+1}$.  

By Lemma \ref{lbz} we have
\begin{eqnarray}\label{ad1x0}&&Ad_{\delta_1}(\xi)(b)=\delta_1\xi(b)-\xi(\delta_1(b))\nonumber\\
&=&\delta_1(\sum_{j=i+1}^w[k](\theta_{l_1}^0\cdots\theta_{l_i}^0\theta_{l_{i+1}}^1\cdots\widehat{\theta^1_{l_j}}\cdots\theta_{l_w}^1)\cdot\xi(\theta_{l_j}^1\rs_m(n+1)))\nonumber\\
&&-\xi(\delta_1(\theta_{l_1}^0\cdots\theta_{l_i}^0\rs_0(n+1))\cdot [1](\theta^1_{l_{i+1}}\cdots\theta_{l_w}^1\rs_m(n+1)))\nonumber\\
&&-\xi(
\sum_{j=i+1}^w[1](\theta_{l_1}^0\cdots\theta_{l_i}^0\cdots \widehat{\theta^1_{l_j}}\cdots\theta_{l_w}^1\rs_m(n+1))\cdot \delta_1(\theta_{l_j}^1\rs_0(n+1))).
\end{eqnarray} 
We have
\begin{eqnarray}\label{ad1x1}&&\delta_1(\sum_{j=i+1}^w[k](\theta_{l_1}^0\cdots\theta_{l_i}^0\theta_{l_{i+1}}^1\cdots\widehat{\theta^1_{l_j}}\cdots\theta_{l_w}^1)\cdot\xi(\theta_{l_j}^1\rs_m(n+1)))\nonumber\\
&=&\sum_{j=i+1}^w[k]\delta_1(\theta_{l_1}^0\cdots\theta_{l_i}^0\theta_{l_{i+1}}^1\cdots\widehat{\theta^1_{l_j}}\cdots\theta_{l_w}^1\rs_0(n+1))\cdot[1]\xi(\theta_{l_j}^1\rs_m(n+1)))\nonumber\\
&&+\sum_{j=i+1}^w[k+1](\theta_{l_1}^0\cdots\theta_{l_i}^0\theta_{l_{i+1}}^1\cdots\widehat{\theta^1_{l_j}}\cdots\theta_{l_w}^1)\cdot\delta_1\xi(\theta_{l_j}^1\rs_m(n+1)))\nonumber\\
&=&\sum_{j=i+1}^w[k]\delta_1(\theta_{l_1}^0\cdots\theta_{l_i}^0\rs_0(n+1))[k+1](\theta_{l_{i+1}}^1\cdots\widehat{\theta^1_{l_j}}\cdots\theta_{l_w}^1)\cdot[1]\xi(\theta_{l_j}^1\rs_m(n+1)))\nonumber\\
&&+\sum_{\substack{j,t=i+1\\ j\neq t}}^{w}[k+1](\theta_{l_1}^0\cdots\widehat{\theta^1_{l_j}}\cdots\widehat{\theta^1_{l_t}}\cdots\theta_{l_w}^1)[k](\delta_1(\theta_{l_t}^1\rs_0(n+1)))\cdot[1]\xi(\theta_{l_j}^1\rs_m(n+1)))\nonumber\\
&&+\sum_{j=i+1}^w[k+1](\theta_{l_1}^0\cdots\theta_{l_i}^0\theta_{l_{i+1}}^1\cdots\widehat{\theta^1_{l_j}}\cdots\theta_{l_w}^1)\cdot\delta_1\xi(\theta_{l_j}^1\rs_m(n+1)))\nonumber\\
\end{eqnarray} 
and
\begin{eqnarray}\label{ad1x3}
&&\xi(
\sum_{j=i+1}^w[1](\theta_{l_1}^0\cdots\theta_{l_i}^0\cdots \widehat{\theta^1_{l_j}}\cdots\theta_{l_w}^1\rs_m(n+1))\cdot \delta_1(\theta_{l_j}^1\rs_0(n+1)))\nonumber\\
&=&\sum_{j=i+1}^w\xi([1]((\theta_{l_1}^0\cdots\theta_{l_i}^0\cdots \widehat{\theta^1_{l_j}}\cdots\theta_{l_w}^1\rs_m(n+1)\cdot (l_j\sum_{a\geq 0}\theta_{l_j-2-a}^1\theta_a^0))))\nonumber\\
&=&\sum_{\substack{j,t=i+1\\ j\neq t}}^w[k+1](\theta_{l_1}^0\cdots\theta_{l_i}^0\cdots \widehat{\theta^1_{l_j}}\cdots\widehat{\theta^1_{l_t}}\cdots\theta_{l_w}^1(l_j\sum_{a\geq 0}\theta_{l_j-2-a}^1\theta_a^0))[1]\xi(\theta^1_{l_t}\rs_m(n+1))\nonumber\\
&&+\sum_{j=i+1}^w[k+1](\theta_{l_1}^0\cdots\theta_{l_i}^0\cdots \widehat{\theta^1_{l_j}}\cdots\theta_{l_w}^1)\cdot\xi( [1](l_j\sum_{a\geq 0}\theta_{l_j-2-a}^1\theta_a^0)\rs_m(n+1))\nonumber\\
&=&\sum_{\substack{j,t=i+1\\ j\neq t}}^w[k+1](\theta_{l_1}^0\cdots\theta_{l_i}^0\cdots \widehat{\theta^1_{l_j}}\cdots\widehat{\theta^1_{l_t}}\cdots\theta_{l_w}^1(l_j\sum_{a\geq 0}\theta_{l_j-2-a}^1\theta_a^0))[1]\xi(\theta^1_{l_t}\rs_m(n+1))\nonumber\\
&&+\sum_{j=i+1}^w[k+1](\theta_{l_1}^0\cdots\theta_{l_i}^0\cdots \widehat{\theta^1_{l_j}}\cdots\theta_{l_w}^1)\cdot\xi\delta_1(\theta_{l_j}^1\rs_m(n+1))\nonumber\\
\end{eqnarray} 
and also
\begin{eqnarray}\label{ad1x2}&&\xi(\delta_1(\theta_{l_1}^0\cdots\theta_{l_i}^0\rs_0(n+1))\cdot [1](\theta^1_{l_{i+1}}\cdots\theta_{l_w}^1\rs_m(n+1)))\nonumber\\
&=&[k]\delta_1(\theta_{l_1}^0\cdots\theta_{l_i}^0\rs_0(n+1))\cdot [1]\xi(\theta^1_{l_{i+1}}\cdots\theta_{l_w}^1\rs_m(n+1))\nonumber\\
&=&\sum_{j=i+1}^w[k]\delta_1(\theta_{l_1}^0\cdots\theta_{l_i}^0\rs_0(n+1))\cdot [k+1](\theta_{l_{i+1}}^1\cdots\widehat{\theta^1_{l_j}}\cdots\theta_{l_w}^1)\cdot[1]\xi(\theta_{l_j}^1\rs_m(n+1)))\nonumber\\
\end{eqnarray}
Combine (\ref{ad1x0}), (\ref{ad1x1}), (\ref{ad1x2}) and (\ref{ad1x3}) and we get 

\[Ad_{\delta_1}(\xi)(b)=\sum_{j=i+1}^w[k+1](\theta_{l_1}^0\cdots\theta_{l_i}^0\theta_{l_{i+1}}^1\cdots\widehat{\theta^1_{l_j}}\cdots\theta_{l_w}^1)\cdot Ad_{\delta_1}(\xi)(\theta_{l_j}^1\rs_m(n+1)).\]
Hence $Ad_{\delta_1}(\xi)$ is a $(1,0)$-operator.

Now we compute $Ad_{\delta_0}(\xi)$.  By Lemma \ref{lbz} we have

\begin{eqnarray}\label{ad0x0}&&Ad_{\delta_0}(\xi)(b)=\delta_0\xi(b)-\xi(\delta_0(b))\nonumber\\
&=&\delta_0(\sum_{j=i+1}^w[k](\theta_{l_1}^0\cdots\theta_{l_i}^0\theta_{l_{i+1}}^1\cdots\widehat{\theta^1_{l_j}}\cdots\theta_{l_w}^1)\cdot\xi(\theta_{l_j}^1\rs_m(n+1)))\nonumber\\
&&-\xi([1](\theta_{l_1}^0\cdots\theta_{l_i}^0\theta_{l_{i+1}}^1\cdots\theta_{l_w}^1)\delta_0(\rs_m(n+1)))\nonumber\\
&=&\sum_{j=i+1}^w[k+1](\theta_{l_1}^0\cdots\theta_{l_i}^0\theta_{l_{i+1}}^1\cdots\widehat{\theta^1_{l_j}}\cdots\theta_{l_w}^1)\cdot\delta_0\xi(\theta_{l_j}^1\rs_m(n+1)))\nonumber\\
&&-\xi([1](\theta_{l_1}^0\cdots\theta_{l_i}^0\theta_{l_{i+1}}^1\cdots\theta_{l_w}^1\sum_{a\geq 0}\theta_a^1\rs_{m-a-1}(n+1)))\nonumber\\
&=&\sum_{j=i+1}^w[k+1](\theta_{l_1}^0\cdots\theta_{l_i}^0\theta_{l_{i+1}}^1\cdots\widehat{\theta^1_{l_j}}\cdots\theta_{l_w}^1)\cdot\delta_0\xi(\theta_{l_j}^1\rs_m(n+1)))\nonumber\\
&&-\sum_{j=i+1}^w[k+1](\theta_{l_1}^0\cdots\theta_{l_i}^0\theta_{l_{i+1}}^1\cdots\widehat{\theta_{l_j}^1}\cdots\theta_{l_w}^1\cdot(\sum_{a\geq 0}\theta_a^1))\xi[1](\theta_{l_j}^1\rs_{m-a-1}(n+1)))\nonumber\\
&&-[k+1](\theta_{l_1}^0\cdots\theta_{l_i}^0\theta_{l_{i+1}}^1\cdots\theta_{l_w}^1)\xi[1](\sum_{a\geq 0}\theta_a^1\rs_{m-a-1}(n+1))\nonumber
\end{eqnarray} 

Because $(1)\xi=\xi(1)$, we have
\[\xi[1](\theta_{l_j}^1\rs_{m-a-1}(n+1))=[1](-a-1)\xi(\theta_{l_j}^1\rs_{m}(n+1)).\]

On the other hand, by Lemma \ref{lbz} we have
\[\forall b'\in B(n),~\delta_0(b')=\sum_{a\geq 0}\theta_a^1(n-1)\cdot [1] ((-a-1)(b')).\]

Therefore
\begin{eqnarray}&&(\sum_{a\geq 0}\theta_a^1(n-k))\xi[1](\theta_{l_j}^1\rs_{m-a-1}(n+1)))\nonumber\\
&=&(\sum_{a\geq 0}\theta_a^1(n-k))\cdot(-a-1)[1]\xi(\theta_{l_j}^1\rs_{m}(n+1)))=\delta_0\xi((\theta_{l_j}^1\rs_{m}(n+1)).\nonumber\end{eqnarray}

Hence

\begin{eqnarray}\label{ad0x1}&&Ad_{\delta_0}(\xi)(b)=\delta_0\xi(b)-\xi(\delta_0(b))\nonumber\\
&=&-[k+1](\theta_{l_1}^0\cdots\theta_{l_i}^0\theta_{l_{i+1}}^1\cdots\theta_{l_w}^1)\xi[1](\sum_{a\geq 0}\theta_a^1\rs_{m-a-1}(n+1))\nonumber\\
&=&-[k+1](\theta_{l_1}^0\cdots\theta_{l_i}^0\theta_{l_{i+1}}^1\cdots\theta_{l_w}^1)\xi\delta_0(\rs_{m-a-1}(n+1))\nonumber\\
&=&[k+1](\theta_{l_1}^0\cdots\theta_{l_i}^0\theta_{l_{i+1}}^1\cdots\theta_{l_w}^1)Ad_{\delta_0}(\xi)(\rs_{m-a-1}(n+1)),
\end{eqnarray} 
where the last equality is because by definition $\xi(\rs_m(n+1))=0$.

The case with $\xi$ a $(2,0)$-operator is analogous and left to the readers.  The lemma is proved. 
\end{proof}
The following corollary follows straightforward from Lemma \ref{adop}.
\begin{coro}\label{oper}$Ad_{\delta_0}(Ad_{\delta_1})^k(\delta_2)$ and $Ad_{\delta_0}(Ad_{\delta_1})^{k-1-s}Ad_{\delta_0}(Ad_{\delta_1})^s(\delta_3), ~s=0,\cdots,k-1$ are all $0$-operators of degree $k+2$.
\end{coro}

Define $\Xi(k):=\left\{Ad_{\delta_0}(Ad_{\delta_1})^k(\delta_2)+\sum_{s=0}^{k-1}Ad_{\delta_0}(Ad_{\delta_1})^{k-1-s}Ad_{\delta_0}(Ad_{\delta_1})^s(\delta_3)\right\}$, then $\Xi(k)$ is a 0-operator of degree $k+2$.  By Proposition \ref{vkmain} we have for any $n,m,k\geq 0$ 
$$\Xi(k)(\rs_m(n))\in(d-3)A^*(S^{[n-k-2]}).$$  

Therefore we have the following corollary.
\begin{coro}\label{vanxi}For any $\xi_1,\xi_2:\bigoplus_{n\geq 0}B(n)\ra\bigoplus_{n\geq0}B(n)$ two group homomorphisms and $b'\in B(n)$, we have 
\[\xi_1\cdot\Xi(k)\cdot\xi_2(b') \in (d-3)\oplus_{n\geq 0}A^*(S^{[n]}).\]
In particular for any $n,m,k,l\geq0$ and any $i_1,\cdots,i_l\in\{0,1,2,3\}$
\[Ad_{\delta_{i_1}}\cdots Ad_{\delta_{i_l}}(\Xi(k))(\rs_m(n))\in (d-3)A^*(S^{[n-k-2-l]}).\]
\end{coro}

\subsection{The final proof.}
We have the following proposition.
\begin{prop}\label{midmain}For any $r,k\geq 0$, we have
\[\left(\delta_2(\sum_{\sum_{l=1}^ki_l=r}\delta_{i_1}\cdots\delta_{i_k})+\delta_3(\sum_{\sum_{l=1}^kj_l=r-1}\delta_{j_1}\cdots\delta_{j_k})\right)\rs_m(n)\in (d-3)A^*(S^{[n-k-1]}).\]
\end{prop}
Theorem \ref{main} is an easy corollary to Proposition \ref{midmain}.
\begin{proof}[Proof of Theorem \ref{main}]We do induction on $k$.  We have
\begin{eqnarray}&&\sum_{\sum_{l=1}^{k+1}i_l=k+1}\delta_{i_1}\cdots\delta_{i_{k+1}}(\rs_m(n))\nonumber\\
&=&\delta_0\cdot\left(\sum_{\sum_{l=1}^{k}s_l=k+1}\delta_{s_1}\cdots\delta_{s_{k}}\right)(\rs_m(n))+\delta_1\cdot\left(\sum_{\sum_{l=1}^{k}s_l=k}\delta_{s_1}\cdots\delta_{s_{k}}\right)(\rs_m(n))\nonumber\\
&&+\delta_2\cdot\left(\sum_{\sum_{l=1}^{k}s_l=k-1}\delta_{s_1}\cdots\delta_{s_{k}}\right)(\rs_m(n))+\delta_3\cdot\left(\sum_{\sum_{l=1}^{k}s_l=k-2}\delta_{s_1}\cdots\delta_{s_{k}}\right)(\rs_m(n)).\nonumber
\end{eqnarray}
By Remark \ref{lag}, we have 
$$\delta_0\left(\sum_{\sum_{l=1}^{k}s_l=k+1}\delta_{s_1}\cdots\delta_{s_{k}}\right)(\rs_m(n))=\delta_0\left(\sum_{\sum_{l=1}^{k}s_l=k+1}\delta_{s_1}\cdots\delta_{s_{k}}(\rs_m(n))\right)=\delta_0(0)=0.$$

By induction assumption,  we have
$$\delta_1\cdot\left(\sum_{\sum_{l=1}^{k}s_l=k}\delta_{s_1}\cdots\delta_{s_{k}}\right)(\rs_m(n))=\delta_1\cdot\left(\sum_{\sum_{l=1}^{k}s_l=k}\delta_{s_1}\cdots\delta_{s_{k}}(\rs_m(n))\right)\in (d-3)A^*(S^{[n-k]}).$$

Finally by Proposition \ref{midmain}, we have
\begin{eqnarray}&&\delta_2\cdot\left(\sum_{\sum_{l=1}^{k}s_l=k-1}\delta_{s_1}\cdots\delta_{s_{k}}\right)(\rs_m(n))\nonumber\\
&+&\delta_3\cdot\left(\sum_{\sum_{l=1}^{k}s_l=k-2}\delta_{s_1}\cdots\delta_{s_{k}}\right)(\rs_m(n))\in (d-3)A^*(S^{[n-k]}).\nonumber
\end{eqnarray}
Hence the theorem.
\end{proof}
Now we only need to prove Proposition \ref{midmain}.  Denote 
$$\rii(k,r):=\delta_2(\sum_{\sum_{l=1}^ki_l=r}\delta_{i_1}\cdots\delta_{i_k})+\delta_3(\sum_{\sum_{l=1}^kj_l=r-1}\delta_{j_1}\cdots\delta_{j_k});$$
and 
$$\ri(k,r):=\delta_2(\sum_{\substack{i_1,\cdots,i_k\in\{0,1\}\\\sum_{l=1}^ki_l=r}}\delta_{i_1}\cdots\delta_{i_k})+\delta_3(\sum_{\substack{j_1,\cdots,j_k\in\{0,1\}\\ \sum_{l=1}^kj_l=r-1}}\delta_{j_1}\cdots\delta_{j_k}).$$
\begin{prop}\label{simain}The following two statements are equivalent:

(1) For any $r,k\geq 0$, $\rii(k,r)(\rs_m(n))\in (d-3)A^*(S^{[n-k-1]})$;

(2) For any $r,k\geq 0$, $\ri(k,r)(\rs_m(n))\in (d-3)A^*(S^{[n-k-1]})$.

In other words, Proposition \ref{midmain} is equivalent to say that for any $r,k\geq 0$
\[\left(\delta_2(\sum_{\substack{i_1,\cdots,i_k\in\{0,1\}\\\sum_{l=1}^ki_l=r}}\delta_{i_1}\cdots\delta_{i_k})+\delta_3(\sum_{\substack{j_1,\cdots,j_k\in\{0,1\}\\ \sum_{l=1}^kj_l=r-1}}\delta_{j_1}\cdots\delta_{j_k})\right)(\rs_m(n))\in(d-3)A^*(S^{[n-k-1]})\]
\end{prop}
\begin{proof}By a direct observation, one can see that
\begin{eqnarray}\rii(k,r)&=&\sum_{\substack{0\leq m\leq r-2\\ 0\leq s\leq k-1}}\sum_{\sum_{l=1}^st_l=m}\delta_2\delta_{t_1}\cdots\delta_{t_s}\ri(k-s-1,r-m-2)\nonumber\\
&+&\sum_{\substack{0\leq m\leq r-3\\ 0\leq s\leq k-1}}\sum_{\sum_{l=1}^st_l=m}\delta_3\delta_{t_1}\cdots\delta_{t_s}\ri(k-s-1,r-m-3)+\ri(k,r);\nonumber
\end{eqnarray}
\begin{eqnarray}\ri(k,r)&=&\rii(k,r)-\sum_{\substack{0\leq m\leq r-2\\ 0\leq s\leq k-1}}\sum_{\substack{t_1,\cdots,t_s\in\{0,1\} \\\sum_{l=1}^st_l=m}}\delta_2\delta_{t_1}\cdots\delta_{t_s}\rii(k-s-1,r-m-2)\nonumber\\
&-&\sum_{\substack{0\leq m\leq r-3\\ 0\leq s\leq k-1}}\sum_{\substack{t_1,\cdots,t_s\in\{0,1\} \\\sum_{l=1}^st_l=m}}\delta_3\delta_{t_1}\cdots\delta_{t_s}\rii(k-s-1,r-m-3)\nonumber\end{eqnarray}
Hence the proposition.
\end{proof}
Define
\begin{eqnarray}\rj(k,r,s)&:=&\sum_{\substack{i_1,\cdots,i_k\in\{0,1\}\\\sum_{l=1}^ki_l=r}}Ad_{\delta_{i_1}}\cdots Ad_{\delta_{i_s}}(\delta_2)\cdot\delta_{i_{s+1}}\cdots\delta_{i_k}\nonumber\\
&&+\sum_{\substack{j_1,\cdots,j_k\in\{0,1\}\\ \sum_{l=1}^kj_l=r-1}}Ad_{\delta_{j_1}}\cdots Ad_{\delta_{j_s}}(\delta_3)\cdot\delta_{j_{s+1}}\cdots\delta_{j_k}.\nonumber\end{eqnarray}
We let $\rj(k,r,s)=0$ if $s<0$ or $s>k$.  Notice that $\rj(k,r,0)=\ri(k,r)$.
\begin{lemma}\label{cad}Fix $0\leq s_0\leq k_0$.  If we assume $\rj(k,r,s_0)(\rs_m(n))\in (d-3)A^*(S^{[n-k-1]})$ for all $r\geq 0$ and $k\leq k_0$, then the following two statements are equivalent:

(1) For any $r\geq 0$, $\rj(k_0+1,r,s_0)(\rs_m(n))\in (d-3)A^*(S^{[n-k-1]})$;

(2) For any $r\geq 0$, $\rj(k_0+1,r,s_0+1)(\rs_m(n))\in (d-3)A^*(S^{[n-k-1]})$.
 \end{lemma}
 \begin{proof}By a direct observation we have
 \begin{eqnarray}&&\rj(k_0+1,r,s_0+1)(\rs_m(n))\nonumber\\
 &=&\delta_0\rj(k_0,r,s_0)(\rs_m(n))+\delta_1\rj(k_0,r-1,s_0)(\rs_m(n))-\rj(k_0+1,r,s_0)(\rs_m(n)).\nonumber\end{eqnarray}
 Hence the lemma.
 \end{proof}
 \begin{proof}[Proof of Proposition \ref{midmain}]By Proposition \ref{simain} it is enough to prove 
\begin{equation}\label{res}\ri(k,r)(\rs_m(n))\in (d-3)A^*(S^{[n-k-1]}).\end{equation} 
 
We do the induction on $k$.  For $k=0$ (\ref{res}) is obvious.  Assume (\ref{res}) holds for $k\leq k_0$, then by Lemma \ref{cad} we have 
\[\rj(k,r,1)(\rs_m(n))\in (d-3)A^*(S^{[n-k-1]}),~\forall k\leq k_0.\]
By applying Lemma \ref{cad} repeatedly, we have 
\[\rj(k,r,s)(\rs_m(n))\in (d-3)A^*(S^{[n-k-1]}),~\forall k,s\leq k_0;\] 
and moreover (\ref{res}) for $k=k_0+1$ is equivalent to the following equation
\begin{equation}\label{rese}\rj(k_0+1,r,k_0+1)(\rs_m(n))\in (d-3)A^*(S^{[n-k_0-2]}).\end{equation}

Recall that 
$$\Xi(k)=\left\{Ad_{\delta_0}(Ad_{\delta_1})^k(\delta_2)+\sum_{s=0}^{k-1}Ad_{\delta_0}(Ad_{\delta_1})^{k-1-s}Ad_{\delta_0}(Ad_{\delta_1})^s(\delta_3)\right\}.$$Then we have 
\begin{eqnarray}&&\rj(k_0+1,r,k_0+1)\nonumber\\&=&\sum_{\substack{i_1,\cdots,i_{k_0+1}\in\{0,1\}\\\sum_{l=1}^{k_0+1}i_l=r}}Ad_{\delta_{i_1}}\cdots Ad_{\delta_{i_{k_0+1}}}(\delta_2)+\sum_{\substack{j_1,\cdots,j_{k_0+1}\in\{0,1\}\\ \sum_{l=1}^{k_0+1}j_l=r-1}}Ad_{\delta_{j_1}}\cdots Ad_{\delta_{j_{k_0+1}}}(\delta_3)\nonumber\\
&=&\sum_{s=0}^{k_0}\sum_{\substack{a_1,\cdots,a_{k_0-s}\in\{0,1\}\\\sum_{l=1}^{k_0-s}a_l=r-s}}Ad_{\delta_{a_1}}\cdots Ad_{\delta_{a_{k_0-s}}}\Xi(s).\nonumber
\end{eqnarray}
The proposition follows from Corollary \ref{vanxi}.
\end{proof}


Yao Yuan\\
Beijing National Center for Applied Mathematics,\\
Academy for Multidisciplinary Studies, \\
Capital Normal University, 100048, Beijing, China\\
E-mail: 6891@cnu.edu.cn.
\end{document}